\documentclass[11pt]{amsart}
\usepackage{amssymb,bbold,mathtools}
\usepackage[all]{xy}
\usepackage{color}
\usepackage{mathrsfs}
\usepackage{mathtools}
\usepackage{enumitem}
\usepackage{hyperref}
\usepackage{cleveref}

\numberwithin{equation}{section}
\allowdisplaybreaks 

\usepackage{array}
\newcolumntype{L}[1]{>{\raggedright\let\newline\\\arraybackslash\hspace{0pt}}m{#1}}
\newcolumntype{C}[1]{>{\centering\let\newline\\\arraybackslash\hspace{0pt}}m{#1}}
\newcolumntype{R}[1]{>{\raggedleft\let\newline\\\arraybackslash\hspace{0pt}}m{#1}}

\usepackage{hhline}

\DeclareMathAlphabet{\mathcalalt}{OMS}{cmsy}{m}{n}

\theoremstyle{plain}
\newtheorem{theorem}{Theorem}[section]
\newtheorem{proposition}[theorem]{Proposition}
\newtheorem{corollary}[theorem]{Corollary}
\newtheorem{lemma}[theorem]{Lemma}

\theoremstyle{definition}
\newtheorem{remark}[theorem]{Remark}
\newtheorem{example}[theorem]{Example}
\newtheorem{examples}[theorem]{Examples}

\numberwithin{table}{theorem}

\newcommand{\abs}[1]{\lvert#1\rvert}
\newcommand{\norm}[1]{\lVert#1\rVert}

\newcommand{\uppars}[1]{\textup{(}#1\textup{)}}

\newcommand{\falg}{$\!f\!$-algebra}
\newcommand{\falgs}{\falg s}
\newcommand{\fsubalg}{$\!f\!$-subalgebra}

\newcommand{\RR}{\mathbb R}

\DeclareMathOperator{\Span}{Span}

\newcommand{\pos}[1]{{#1}^+}
\newcommand{\posE}{\pos{E}}

\newcommand{\ob}{\mathrm{ob}}
\newcommand{\oc}{\mathrm{oc}}

\newcommand{\obops}{\mathcal{L}_{\ob}}
\newcommand{\ocops}{\mathcal{L}_{\oc}}

\newcommand{\opalg}{{\mathcal A}} 
\newcommand{\opalgtwo}{{\mathcal B}} 

\newcommand{\opfont}{\mathcal}

\newcommand{\oplat}{{\opfont E}}
\newcommand{\oplattwo}{{\opfont F}}
\newcommand{\opset}{{\opfont S}}
\newcommand{\alg}{\opfont A}

\newcommand{\convstrufont}{\mathcal}
\newcommand{\convstru}{{\convstrufont C}}

\newcommand{\odual}[1]{{#1}^{\thicksim}}
\newcommand{\ocdual}[1]{{#1}^\thicksim_{\oc}}

\newcommand{\Ell}{\mathrm{L}}

\newcommand{\conv}{\xrightarrow}
\newcommand{\convwithoverset}[1]{\conv{#1}}

\renewcommand{\o}{\ensuremath{\mathrm{o}}}

\newcommand{\uo}{\ensuremath{\mathrm{uo}}}

\newcommand{\STRONG}{\ensuremath{\mathrm{S}}}
\newcommand{\SO}{\ensuremath{\mathrm{SO}}}
\newcommand{\SUO}{\ensuremath{\mathrm{SUO}}}
\newcommand{\uoLt}{{\widehat{\tau}}}
\newcommand{\SuoLt}{{S\widehat{\tau}}}

\newcommand{\oconv}{\convwithoverset{\o}}

\newcommand{\uoconv}{\convwithoverset{\uo}}

\newcommand{\SOconv}{\convwithoverset{\SO}}
\newcommand{\SUOconv}{\convwithoverset{\SUO}}

\newcommand{\tauconv}{\convwithoverset{\tau}}

\newcommand{\orthcom}[1]{{#1}^{\prime}_{\Orth}}

\newcommand{\netgen}[2]{{(#1)_{#2}}}

\newcommand{\indexsetfont}{\mathcalalt}
\newcommand{\is}{\indexsetfont{A}}
\newcommand{\net}{\netgen{x_\alpha}{\alpha\in\is}}
\newcommand{\opnet}{\netgen{S_\alpha}{\alpha\in\is}}

\newcommand{\istwo}{\indexsetfont{B}}
\newcommand{\nettwo}{\netgen{x_\beta}{\beta\in\istwo}}

\newcommand{\opnettwo}{\netgen{T_\beta}{\beta\in\istwo}}

\newcommand{\seq}[2]{(#1)_{#2=1}^\infty}

\newcommand{\Ri}[1]{\rho_{#1}}
\newcommand{\Le}[1]{\lambda_{#1}}

\newcommand{\uoLtop}{uo-Lebes\-gue topology}
\newcommand{\uoLtops}{uo-Lebes\-gue topologies}

\newcommand{\necun}{\textup{(}nec\-es\-sar\-i\-ly unique\textup{)}}

\DeclareMathOperator{\Orth}{Orth}
\newcommand{\Id}{I}

\newcommand{\oadhtext}{o-ad\-her\-ence}
\newcommand{\uoadhtext}{uo-ad\-her\-ence}

\newcommand{\adh}[2]{a_{#1}(#2)}
\newcommand{\sadh}[2]{a_{\sigma #1}(#2)}

\newcommand{\oadh}[1]{\adh{\o}{#1}}
\newcommand{\uoadh}[1]{\adh{\uo}{#1}}
\newcommand{\soadh}[1]{\sadh{\o}{#1}}
\newcommand{\suoadh}[1]{\sadh{\uo}{#1}}

\newcommand{\Soadh}[1]{\adh{\SO}{#1}}
\newcommand{\SUOadh}[1]{\adh{\SUO}{#1}}

\usepackage{datetime}
\usepackage{currfile}

\usepackage[charter]{mathdesign}

\usepackage{float}

\numberwithin{table}{theorem}

\usepackage{chngcntr}
\counterwithin{table}{section}

\usepackage{threeparttable}

\usepackage{caption}
\setlist[enumerate,1]{label=\textup{(\arabic*)},ref=\textup{(\arabic*)}}
\setlist[enumerate,2]{label=\textup{(\alph*)},ref=\textup{(\alph*)}}
\setlist[enumerate,3]{label=\textup{(\roman*)},ref=\textup{(\roman*)}}
\setlist[enumerate,4]{label=\textup{(\Alph*)},ref=\textup{(\Alph*)}}

\newlist{enumerate_arabic}{enumerate}{1}
\setlist[enumerate_arabic,1]{label=\textup{(\arabic*)},ref=\textup{(\arabic*)}}

\newlist{enumerate_alpha}{enumerate}{1}
\setlist[enumerate_alpha,1]{label=\textup{(\alph*)},ref=\textup{(\alph*)}}

\newlist{enumerate_roman}{enumerate}{1}
\setlist[enumerate_roman,1]{label=\textup{(\roman*)},ref=\textup{(\roman*)}}

\newlist{enumerate_Alpha}{enumerate}{1}
\setlist[enumerate_Alpha,1]{label=\textup{(\Alph*)},ref=\textup{(\Alph*)}}

\crefname{theorem}{Theorem}{Theorems}
\crefname{proposition}{Proposition}{Propositions}
\crefname{lemma}{Lemma}{Lemmas}
\crefname{corollary}{Corollary}{Corollaries}
\crefname{conjecture}{Conjecture}{Conjectures}
\crefname{definition}{Definition}{Definitions}
\crefname{example}{Example}{Examples}
\crefname{examples}{Examples}{Examples}
\crefname{remark}{Remark}{Remarks}
\crefname{assumption}{Assumption}{Assumptions}
\crefname{hypothesis}{Hypothesis}{Hypotheses}
\crefname{question}{Question}{Questions}
\crefname{problem}{Problem}{Problems}
\crefname{task}{Task}{Tasks}
\crefname{addendum}{Addendum}{Addenda}
\crefname{idea}{Idea}{Ideas}
\crefname{suggestion}{Suggestion}{Suggestions}
\crefname{context}{Context}{Contexts}
\crefname{exercise}{Exercise}{Exercises}

\crefname{section}{Section}{Sections}
\crefname{subsection}{Section}{Sections}
\crefname{subsubsection}{Section}{Sections}

\crefname{equation}{equation}{equations}
\crefname{enumi}{part}{parts}
\crefname{enumii}{part}{parts}
\crefname{enumiii}{part}{parts}
\crefname{enumiv}{part}{parts}

\begin{document}
\title[Convergence structures and Hausdorff uo-Lebesgue topologies]{
Convergence structures and Hausdorff uo-Lebesgue topologies on vector lattice algebras of operators}

\author{Yang Deng}
\address{Yang Deng; School of Economic Mathematics, Southwestern University of Finance and Economics, Chengdu, Sichuan 611130, People's Republic of China}
\email{dengyang@swufe.edu.cn}

\author{Marcel de Jeu}
\address{Marcel de Jeu; Mathematical Institute, Leiden University, P.O.\ Box 9512, 2300 RA Leiden, the Netherlands;
	and Department of Mathematics and Applied Mathematics, University of Pretoria, Cor\-ner of Lynnwood Road and Roper Street, Hatfield 0083, Pretoria, South Africa}
\email{mdejeu@math.leidenuniv.nl}


\begin{abstract}
	A vector sublattice of the order bounded operators on a Dedekind complete vector lattice can be supplied with the convergence structures of order convergence, strong order convergence, unbounded order convergence, strong unbounded order convergence, and, when applicable, convergence with respect to a Hausdorff uo-Lebesgue topology and strong convergence with respect to such a topology. We determine the general validity of the implications between these six convergences on the order bounded operator and on the orthomorphisms. Furthermore, the continuity of left and right multiplications with respect to these convergence structures on the order bounded operators, on the order continuous operators, and on the orthomorphisms is investigated, as is their simultaneous continuity. A number of results are included on the equality of adherences of vector sublattices of the order bounded operators and of the orthomorphisms with respect to these convergence structures. These are consequences of more general results for vector sublattices of arbitrary Dedekind complete vector lattices.\\
The special attention that is paid to vector sublattices of the orthomorphisms is motivated by explaining their relevance for representation theory on vector lattices.
\end{abstract}


\maketitle

\section{Introduction and overview}\label{sec:introduction_and_overviews}

\noindent
In an earlier paper \cite{deng_de_jeu:2021} the authors studied aspects of locally solid linear topologies on vector lattices of order bounded linear operators between vector lattices. Particular attention was paid to the possibility of introducing a Hausdorff \uoLtop\ on such vector lattices.

Such vector lattices of operators carry two, and possibly three, natural convergence structures (order convergence, unbounded order convergence, and, when applicable, convergence with respect to the Hausdorff \uoLtop), as they can be defined for arbitrary vector lattices. For vector lattices of operators, however, besides these `uniform' convergence structures, there are also two, or possibly three, corresponding `strong' counterparts that can be defined in the obvious way. Several relations between the resulting six convergence structures on vector lattices of operators were also investigated in \cite{deng_de_jeu:2021}. In view of their relevance for representation theory in vector lattices, special emphasis was put on the orthomorphisms on a Dedekind complete vector lattice. In this case, implications between convergences hold that do not hold for more general vector lattices of operators. Furthermore, it was shown that the orthomorphisms are not only order continuous, but also continuous with respect to unbounded order convergence on the vector lattice and with respect to a possible Hausdorff \uoLtop\ on it.

Apart from their intrinsic interest, the results in \cite{deng_de_jeu:2021} can be viewed as a part of the groundwork that has to be done in order to facilitate further developments of aspects of the theory of vector lattices of operators. The questions that are asked are natural and basic, but even so the answers are often more easily formulated than proved.

In the present paper, we take this one step further and study these six convergence structures in the context of vector lattice \emph{algebras} of order bounded linear operators on a Dedekind complete vector lattice. Also here there are many natural questions of a basic nature that need to be answered before one can expect to get much further with the theory of such vector lattice algebras and with representation theory on vector lattices. For example, is the left multiplication by a fixed element continuous on the order bounded linear operators with respect to unbounded order convergence? Is the multiplication on the order continuous linear operators simultaneously continuous with respect to a possible Hausdorff \uoLtop\ on it? Given a vector lattice subalgebra of the order continuous linear operators, is the closure (we shall actually prefer to speak of the `adherence') in the order bounded linear operators with respect to strong unbounded order convergence again a vector lattice subalgebra? Is there a condition, sufficiently lenient to be of practical relevance, under which the order adherence of a vector lattice subalgebra of the orthomorphisms coincides with its closure in a possible Hausdorff \uoLtop? Building on \cite{deng_de_jeu:2021}, we shall answer these questions in the present paper, together with many more similar ones. As indicated, we hope and expect that, apart from their intrinsic interest, this may serve as a stockpile of basic, but non-elementary, results that will facilitate a further development of the theory of vector lattice algebras of operators and of representation theory in vector lattices.

\smallskip

\noindent
This paper is organised as follows.

\cref{sec:preliminaries} contains the necessary notation, definitions, and conventions, as well as a few preparatory results that are of interest in their own right. \cref{res:orthisorth_linear} shows that, in many cases of practical interest, a unital positive linear representation of a unital \falg\  on a vector lattice is always an action by orthomorphisms. Its consequence \cref{res:orthisorth} specialises this to the case of left and right multiplications of order bounded operators by orthomorphisms.

In \cref{sec:implications_between_convergences}, we study the validity of each of the 36 possible implications between the 6 convergences that we consider on vector lattice algebras of order bounded linear operators on a Dedekind complete vector lattice. We do this for the order bounded linear operators as well as for the orthomorphisms. The results that are already in \cite{deng_de_jeu:2021} and a few additional ones are sufficient to complete the \cref{table:convergences_general,table:convergences_orthomorphisms}.

\cref{sec:continuity_in_one_variable} contains our results on the continuity of the left and right multiplications by a fixed element with respect to each of the six convergence structures on the order bounded linear operators. For this, we distinguish between the multiplication by an arbitrary order bounded linear operator, by an order continuous one, and by an orthomorphism. By giving (counter) examples, we show that our results are sharp in the sense that, whenever we state that continuity holds for multiplication by, e.g., an orthomorphism, it is no longer generally true for an arbitrary order continuous linear operator, i.e., for an operator in the `next best class'. We also consider these questions for the orthomorphisms. The results are contained in \cref{table:right_multiplication,table:left_multiplication,table:multiplication_on_orthomorphisms}.

In \cref{sec:simultaneous_continuity}, we investigate the simultaneous continuity of the multiplication with respect to each of the six convergence structures. When there is simultaneous continuity, the adherence of a subalgebra is, of course, again a subalgebra. With only one exception (see \cref{res:Soad_is_alg_VL} and \cref{rem:Soad_is_alg}), we give (counter) examples to show that our conditions for the adherence of an algebra to be a subalgebra again are `sharp' in the sense as indicated above for \cref{sec:continuity_in_one_variable}.

\cref{sec:equality_of_adherences} is dedicated to the equality of various adherences of vector sublattices and vector lattice subalgebras. It is also indicated there how representation theory in vector lattices leads quite naturally to the study of vector lattice subalgebras of the orthomorphisms (see the  \cref{res:commutant_of_bounded_operators_is_lattice,res:commutant_of_order_continuous_operators_is_lattice}), thus motivating in more detail the special attention that is paid to the orthomorphisms in \cite{deng_de_jeu:2021} and in the present paper.

\section{Preliminaries}\label{sec:preliminaries}

\noindent 
In this section, we give the notation, conventions, and definitions used in the sequel. We also include a few preliminary results.

All vector spaces are over the real numbers and all vector lattices are supposed to be Archimedean. We let  $\posE$ denote the positive cone of a vector lattice $E$.
The identity operator on a vector lattice $E$ will be denoted by $\Id$, or by $I_E$ when the context requires this. The characteristic function of a set $S$ is denoted by $\chi_S$.

Let $E$ be a vector lattice, and let $x\in E$. We say that a net $\net$ in $E$ is \emph{order convergent to $x\in E$} (denoted by $x_\alpha\oconv x$) when there exists a net $\nettwo$ in $E$ such that $y_\beta\downarrow 0$ and with the property that, for every $\beta_0\in\istwo$, there exists an $\alpha_0\in \is$ such that $\abs{x-x_\alpha}\leq y_{\beta_0}$ whenever $\alpha$ in $\is$ is such that $\alpha\geq\alpha_0$. Note that, in this definition, the index sets $\is$ and $\istwo$ need not be equal.

A net $\net$ in a vector lattice $E$ is said to be \emph{unbounded order convergent} to an element $x$ in $E$  (denoted by $x_\alpha\uoconv x$) when $\abs{x_\alpha-x}\wedge y\oconv 0$ in $E$ for all $y\in \posE$. Order convergence implies unbounded order convergence to the same limit. For order bounded nets, the two notions coincide.

Let $E$ and $F$ be vector lattices. The order bounded linear operators from $E$ into $F$ will be denoted by $\obops(E,F)$; this is a Dedekind complete vector lattice when $F$ is.  We write $\odual{E}$ for $\obops(E,\mathbb R)$. A linear operator  $T: E\to F$ between two vector lattices $E$ and $F$ is \emph{order continuous} when, for every net $\net$ in $E$,  the fact that $x_\alpha\oconv 0$ in $E$ implies that $Tx_\alpha\oconv 0$ in $F$. An order continuous linear operator between two vector lattices is automatically order bounded; see \cite[Lemma~1.54]{aliprantis_burkinshaw_POSITIVE_OPERATORS_SPRINGER_REPRINT:2006}, for example. The order continuous linear operators from $E$ into $F$ will be denoted by $\ocops(E,F)$. We write $\ocdual{E}$ for $\ocops(E,\RR)$.

Let $F$ be a vector sublattice of a vector lattice $E$. Then $F$ is a \emph{regular vector sublattice of $E$} when the inclusion map from $F$ into $E$ is order continuous. Ideals are regular vector sublattices. For a net in a regular vector sublattice $F$ of $E$, its unbounded order convergence in $F$ and in $E$ are equivalent; see \cite[Theorem~3.2]{gao_troitsky_xanthos:2017}.

An \emph{orthomorphism} on a vector lattice $E$ is a band preserving order bounded linear operator. We let $\Orth(E)$ denote the orthomorphisms on $E$. Orthomorphisms are automatically order continuous; see \cite[Theorem~2.44]{aliprantis_burkinshaw_POSITIVE_OPERATORS_SPRINGER_REPRINT:2006}. An overview of some basic properties of the orthomorphisms that we shall use can be found in the first part of \cite[Section~6]{deng_de_jeu:2021}, with detailed references included.

A topology $\tau$ on a vector lattice $E$ is a \emph{\uoLtop} when it is a (not necessarily Hausdorff) locally solid linear topology on $E$ such that, for a net $\net$ in $E$, the fact that $x_\alpha\uoconv 0$ in $E$ implies that $x_\alpha\tauconv 0$. For the general theory of locally solid linear topologies on vector lattices we refer to \cite{aliprantis_burkinshaw_LOCALLY_SOLID_RIESZ_SPACES_WITH_APPLICATIONS_TO_ECONOMICS_SECOND_EDITION:2003}. A vector lattice need not admit a \uoLtop, and it admits at most one Hausdorff \uoLtop; see \cite[Propositions~3.2,~3.4, and~6.2]{conradie:2005} or \cite[Theorems~5.5 and~5.9]{taylor:2019}). In this case, this unique Hausdorff \uoLtop\ is denoted by $\uoLt_E$.

The following fact will often be used in the present paper.

\begin{theorem}\label{res:all_three_or_none}
	Let $E$ be a Dedekind complete vector lattice. The following are equivalent:
	\begin{enumerate}
		\item\label{res:all_three_or_none1} $E$ admits a \necun\ Hausdorff \uoLtop;
		\item\label{res:all_three_or_none2} $\Orth(E)$ admits a \necun\ Hausdorff \uoLtop;
		\item\label{res:all_three_or_none3} $\obops(E)$ admits a \necun\ Hausdorff \uoLtop.
	\end{enumerate}
\end{theorem}

\begin{proof}
	The equivalence of the parts~\ref{res:all_three_or_none1} and~\ref{res:all_three_or_none2} is a part of \cite[Proposition~8.2]{deng_de_jeu:2021}. Part~\ref{res:all_three_or_none1} implies part~\ref{res:all_three_or_none3} by \cite[Theorem~4.3]{deng_de_jeu:2021}, and part~\ref{res:all_three_or_none3} implies part~\ref{res:all_three_or_none2} by \cite[Proposition~5.12]{taylor:2019}.	
\end{proof}

Let $X$ be a non-empty set. As in \cite{deng_de_jeu:2021}, we define a \emph{convergence structure on $X$} to be a non-empty collection $\convstru$ of pairs $(\net,x)$, where $\net$ is a net in $X$ and $x\in X$, such that:
\begin{enumerate}
	\item\label{part:net_convergence_struture_2} when $(\net,x)\in\convstru$, then also $(\nettwo,x)\in\convstru$ for every subnet $\nettwo$ of $\net$.
	\item\label{part:net_convergence_struture_1} when a net $\net$ in $X$ is constant with value $x$, then $(\net,x)\in\convstru$.
\end{enumerate}

Replacing nets by sequences and subnets by subsequences gives the usual sequential convergence structure, as in \cite[Definition~1.7.1]{beattie_butzmann_CONVERGENCE_STRUCTURES_AND_APPLICATIONS_TO_FUNCTIONAL_ANALYSIS:2002}.

\begin{remark}
Our definition of a convergence structures is actually a possible definition of a so-called net convergence structure. The theory of their counterparts, the so-called filter convergence structures, has been canonised in \cite{beattie_butzmann_CONVERGENCE_STRUCTURES_AND_APPLICATIONS_TO_FUNCTIONAL_ANALYSIS:2002}. It is only recently that a definition of a net convergence structure (more sophisticated than ours) has been given that can be shown to yield a natural bijection between the net convergence structures and the filter convergence structures on a set; see \cite{o'brien_troitsky_van_der_walt_UNPUBLISHED:2021}. In this definition, not all index sets for the nets in the structure are admitted, the admissible index are allowed to be merely pre-ordered, and property~\ref{part:net_convergence_struture_2} in the above definition is replaced with two others. We refer to  \cite{o'brien_troitsky_van_der_walt_UNPUBLISHED:2021} for further details, and content ourselves with our definition above that is sufficient for our merely descriptive purposes.
\end{remark}

Suppose that $\convstru$ is a convergence structure on a non-empty set $X$. For a non-empty subset $S\subseteq X$, we define the \emph{$\convstru$-adherence of $S$ in $X$} as
\[
\adh{\convstru}{S}  \coloneqq \{x\in E\hskip -1pt:\hskip -1pt\! \text{ there exists a net }\net \text{ in } S \text{ such that } (\net,x)\in \convstru\}
\]
We set $\adh{\convstru}{\emptyset}\coloneqq\emptyset$. A subset $S$ of $X$ is said to be \emph{$\convstru$-closed} when $\adh{\convstru}{S}=S$. It is evident how define the adherence of a subset in the case of a sequential convergence structure. The following result, which was mentioned in \cite[Section~1]{deng_de_jeu:2021} without proof (see also \cite[Section~8]{deng_de_jeu:2022} for special cases), was already established in a context of order convergence on partially ordered vector spaces as \cite[Theorem~3.1]{van_imhoff_THESIS:2012}. Its sequential version is proved using a similar argument.

\begin{lemma}
	Let $X$ be a non-empty set, and let $\convstru$ be a convergence structure on $X$. Then the $\convstru$-closed subsets of $X$ are the closed sets of a topology $\tau_{\!\convstru}$ on $X$.
\end{lemma}

\begin{proof}
	It is trivial that $\emptyset$ and $X$ are $\convstru$-closed, and it is immediate that an arbitrary intersection of $\convstru$-closed subsets of $X$ is $\convstru$-closed. We claim that $\adh{\convstru}{S_1\cup S_2}=\adh{\convstru}{S_1}\cup\adh{\convstru}{S_2}$ for arbitrary $S_1,S_2\subseteq X$, which implies that finite unions of $\convstru$-closed subsets are again $\convstru$-closed. Since it is obvious that $\adh{\convstru}{S_1\cup S_2}\supseteq\adh{\convstru}{S_1}\cup\adh{\convstru}{S_2}$, we need to show only the reverse inclusion.  We may suppose that $S_1,S_2\neq\emptyset$. Take an $x\in\adh{\convstru}{S_1\cup S_2}$. Then there exists a net $\net$ in $S_1\cup S_2$ such that $(\net,x)\in\convstru$. If there is a tail $(x_\alpha)_{\alpha\geq\alpha_0}$ that is contained in $S_1$, then it follows from property~\ref{part:net_convergence_struture_2} of $\convstru$ that $x\in\adh{\convstru}{S_1}$. If no tail of $\net$ is contained in $S_1$, then the set $\istwo\coloneqq\{\beta\in\is: x_\beta\in S_2\}$ is a co-final subset of $\is$. Hence $\nettwo$ can canonically be viewed as a subnet of $\net$ that is contained in $S_2$, and property~\ref{part:net_convergence_struture_2} implies that $x\in\adh{\convstru}{S_2}$. In both cases, $x\in\adh{\convstru}{S_1}\cup\adh{\convstru}{S_2}$.
\end{proof}

It is not generally true that $\adh{\convstru}{S}$ is $\tau_{\!\convstru}$-closed. We have the following result, the final statement of which was already mentioned without proof for special cases in \cite[Section~8]{deng_de_jeu:2022}. Its sequential version is valid by essentially the same proof.

\begin{lemma}
	Let $X$ be a non-empty set, let $\convstru$ be a convergence structure on $X$, and let $S\subseteq X$. Then
	\[
	S\subseteq \adh{\convstru}{S}\!\subseteq \overline{S}^{\tau_{\!\convstru}}s.
	\]
	Consequently, $\overline{\adh{\convstru}{S}}^{\tau_{\!\convstru}}\!=\overline{S}^{\tau_{\!\convstru}}$. Furthermore, $\adh{\convstru}{S}$ is $\tau_{\!\convstru}$-closed if and only if $\adh{\convstru}{S}=\overline{S}^{\tau_{\!\convstru}}$.
\end{lemma}

\begin{proof}
Property~\ref{part:net_convergence_struture_1} of $\convstru$ implies that $S\subseteq\adh{\convstru}{S}$, and the obvious monotony of the adherence set map implies that $\overline{S}^{\tau_{\!\convstru}}=\adh{\convstru}{\overline{S}^{\tau_{\!\convstru}}}\supseteq \adh{\convstru}{S}$.  The remaining two statements follow easily from the chain of inclusions.
\end{proof}

On a vector lattice $E$, the set of all pairs of order convergent nets and their order limits forms a convergence structure $\convstru_\o$ on $E$. Likewise, there is a convergence structure $\convstru_{\uo}$ on $E$ and, when applicable, a topological convergence structure $\convstru_{\uoLt_E}$. For a subset $S$ of $E$, we shall write $\oadh{S}$ for $\adh{{\convstru_\o}}{S}$,  $\uoadh{S}$ for $\adh{{\convstru_\uo}}{S}$,  and, when applicable, $\overline{S}^{\uoLt_E}$ for $\adh{\convstru_{\uoLt_E}}{S}$. The corresponding sequential convergence structures are denoted by $\convstru_{\sigma\o}$, $\convstru_{\sigma\uo}$, and, when applicable, $\convstru_{\sigma\uoLt_E}$, respectively. There are self-explanatory notations $\soadh{S}$, $\suoadh{S}$, and, when applicable, $\adh{\sigma\uoLt_E}{S}$. We shall also speak of the order adherence (or \o-adherence) of a subset, rather than of its $\convstru_\o$-adherence; etc.  Note that the order adherence $\oadh{S}$ of $S$ is what is called the `order closure' of $S$ in other sources. Since this `order closure' need not be closed in the $\tau_{\convstru_\o}$-topology on $E$, we shall not use this terminology that is prone to mistakes.

Let $E$ and $F$ be vector lattices, where $F$ is Dedekind complete. Suppose that $\oplat$ is a vector sublattice of $\obops(E,F)$. As for general vector lattices, we have the convergence structures $\convstru_\o (\oplat)$, $\convstru_{\uo}(\oplat)$ and, when applicable, a convergence structure $\convstru_{\uoLt_\oplat}$ on $\oplat$. In addition to these `uniform' convergence structures, there are in this case also `strong' ones that we shall now define. Let $\netgen{T_\alpha}{\alpha\in\is}$ be a net in $\oplat$, and let $T\in\oplat$. Then we shall say that $\netgen{T_\alpha}{\alpha\in\is}$ is \emph{strongly order convergent to $T$} (denoted by $T_\alpha\SOconv T$) when $T_\alpha x\oconv T x$ for all $x\in E$. The set of all pairs of strongly order convergent nets in $\oplat$ and their limits forms a convergence structure $\convstru_{\SO}$ on $\oplat$. Likewise, the net is \emph{strongly unbounded order convergent to $T$} (denoted by $T_\alpha\SUOconv T$) when it is pointwise unbounded order convergent to $T$, resulting in a convergence structure $\convstru_{\SUO}$ on $\oplat$. When $E$ admits a Hausdorff \uoLtop\ $\uoLt_E$, then a net is \emph{strongly convergent with respect to $\uoLt_E$ to $T$} (denoted by $T_\alpha\conv{\STRONG\uoLt_E} T$) when it is pointwise $\uoLt_E$ convergent to $T$, yielding to a convergence structure $\convstru_{S\uoLt_{\oplat}}$ on $\oplat$. As for the three convergence structures on general vector lattices, we shall simply write $\SUOadh{\mathcal S}$ for the $\convstru_{\SUO}$-adherence $\adh{\convstru_{\SUO}}{\mathcal S}$ of a subset $\mathcal S$ of $\oplat$; etc. We shall use a similar simplified notation for adherences corresponding to the sequential strong convergence structures that are defined in the obvious way.

The adherence of a set in a convergence structure obviously depends on the superset, since this determines the available possible limits of nets. In an ordered context, there can be additional complications because, for example, the notion of order convergence of a net itself depends on the vector lattice that the net is considered to be a subset of. It is for this reason that, although we have not included the superset in the notation for adherences, we shall always indicate it in words.

Let $\convstru_X$ be a convergence structure on a non-empty set $X$, and let $\convstru_Y$ be a convergence structure on a non-empty set $Y$. A map $\Phi: X\to Y$ is said to be \emph{$\convstru_X$-$\convstru_Y$ continuous} when, for every pair $(\net,x)$ in $\convstru_X$, the pair $(\netgen{\Phi (x_\alpha)}{\alpha\in\is},\Phi(x))$ is an element of $\convstru_Y$. We shall speak of  $\STRONG\uoLt_E$-\o\ continuity rather than of $\convstru_{\STRONG\uoLt_E}$-$\convstru_{\o}$ continuity; etc.

Let $E$ be a vector lattice. For $T\in \obops(E)$, we define $\Ri{T}, \Le{T}:\obops(E)\to\obops(E)$ by setting $\Ri{T} (S)\coloneqq ST$ and $\Le{T} (S)\coloneqq TS$ for $S\in\obops(E)$. We shall use the same notations for the maps that result in other contexts when compositions with linear operators map one set of linear operators into another.

For later use in this paper, we establish a few preparatory results that are of some interest in their own right.

\begin{lemma}\label{res:orthisorth_basic}
	Let $\opalg$ be an \falg\ with an identity element $e$, and let $E$ be a vector lattice with the principal projection property. Let $a\in\pos{\alg}$, and suppose that
	\[
	\pi:\Span\{e, a, a^2\}\to\obops(E)
	\]
	is a positive linear map such that $\pi(e)=\Id$. Then $\pi(a)\in\Orth(E)$.
\end{lemma}

\begin{proof}
	It is obvious that $\pi(a)\in\obops(E)$, so it remains to be shown that $\pi(a)$ is band preserving on $E$. We know from \cite[Theorem~2.57]{aliprantis_burkinshaw_POSITIVE_OPERATORS_SPRINGER_REPRINT:2006} that
	\[
	a\leq a\wedge ne+\frac{1}{n}a^2\leq ne+\frac{1}{n}a^2
	\]
	for $n\geq 1$. Take $x\in \pos{E}$. Then we have
	\begin{equation}\label{eq:f-algebra_inequality}
		\pi(a)x\leq \pi\left[ne+\frac{1}{n}a^2\right]x= n x+\frac{1}{n}\pi(a^2)x.
	\end{equation}
	for $n\geq 1$.
	Let $B_x$ be the band generated by $x$ in $E$, and let $P_x\in\obops(E)$ be the order projection onto $B_x$. Using that  $\pi(a)x\geq 0$ and \cref{eq:f-algebra_inequality}, we have
	\begin{align*}
		0&\leq (I-P_x)[\pi(a)x]
		\\&\leq (I-P_x)[nx+\frac{1}{n}\pi(a^2)x]
		\\&=\frac{1}{n}(I-P_x)[\pi(a^2)x]
	\end{align*}
	for all $n\geq 1$. Hence $(I-P_x)[\pi(a)x]=0$, so that $\pi(a)x\in B_x$. Since $x$ was arbitrary, this shows that $\pi(a)$ is band preserving.
\end{proof}

\begin{proposition}\label{res:orthisorth_linear}
	Let $\opalg$ be an \falg\ with an identity element $e$, and let $E$ be a vector lattice with the principal projection property. Suppose that $\pi:\opalg\to\obops(E)$ is a positive linear map such that $\pi(e)=\Id$. Then $\pi(\opalg)\subseteq\Orth(E)$. If, in addition, $\pi$ preserves the multiplication, then $\pi$ is a unital vector lattice algebra homomorphism from $\opalg$ into $\Orth(E)$.
\end{proposition}

\begin{proof}
It is clear from \cref{res:orthisorth_basic} that $\pi$ maps $\opalg$ into $\Orth(E)$. Suppose that $\pi$ also preserves the multiplication. In this case, we note that $\Orth(E)$ is a unital Archimedean \falg\ (see \cite[Theorem~2.59]{aliprantis_burkinshaw_POSITIVE_OPERATORS_SPRINGER_REPRINT:2006}), so that it is semiprime by \cite[Corollary~10.4]{de_pagter_THESIS:1981}. Since $\alg$ is likewise semiprime, it follows from \cite[p.~96]{de_pagter_THESIS:1981} (see also \cite[part~(i) of Theorem~3.7]{de_pagter_THESIS:1981}) that $\pi$ is a vector lattice homomorphism.
\end{proof}


The following is immediate from \cref{res:orthisorth_linear} and\textemdash for its first part\textemdash the commutativity of $\Orth(E)$ (see \cite[Theorem~2.56]{aliprantis_burkinshaw_POSITIVE_OPERATORS_SPRINGER_REPRINT:2006}).

\begin{corollary}\label{res:orthisorth}
	Let $E$ and $F$ be vector lattices, where $F$ is Dedekind complete. Let $\oplat$ be a vector sublattice of $\obops(E,F)$ with the principal projection property.
	\begin{enumerate}
		\item\label{res:orthisorth1} Suppose that $ST\in\oplat$ for all $S\in\oplat$ and $T\in\Orth(E)$, so that there is a naturally defined map $\Ri{T}:\oplat\to\oplat$ for $T\in\Orth(E)$. Then $\Ri{T}\in\Orth(\oplat)$ for $T\in\Orth(E)$, and the ensuing map $\Ri{}:\Orth(E)\to\Orth(\oplat)$ is a unital vector lattice algebra homomorphism.
		\item\label{res:orthisorth2} Suppose that $TS\in\oplat$ for all $S\in\oplat$ and $T\in\Orth(F)$, so that there is a naturally defined map $\Le{T}:\oplat\to\oplat$ for $T\in\Orth(F)$. Then $\Le{T}\in\Orth(\oplat)$ for $T\in\Orth(F)$, and the ensuing map $\Le{}:\Orth(F)\to\Orth(\oplat)$ is a unital vector lattice algebra homomorphism.
	\end{enumerate}
\end{corollary}

\begin{remark}\quad
	\begin{enumerate}
		\item For $\oplat=\obops(E,F)$, \cref{res:orthisorth} is established in the beginning of \cite[Section~2]{luxemburg_de_pagter:2002}.
		\item For $\oplat=\ocops(E)$, where $E$ is a Dedekind complete vector lattice, the facts that left and right multiplications by orthomorphisms on $E$ yield orthomorphisms on $\ocops(E)$, were established in \cite[Proof of Theorem~8.4]{de_rijk_THESIS:2012}.
		\item For $\oplat=\Orth(E)$, \cite[Theorems~2.59 and ~2.62]{aliprantis_burkinshaw_POSITIVE_OPERATORS_SPRINGER_REPRINT:2006} show that the (then coinciding) maps $\Ri{}$ and $\Le{}$ even provide a vector lattice algebra isomorphisms between the $\Orth(E)$ and $\Orth(\Orth(E)$. This is also true when $E$ is not Dedekind complete.
	\end{enumerate}
\end{remark}

\section{Implications between convergences on vector lattices of operators}\label{sec:implications_between_convergences}

\noindent In this section, we investigate the implications between the six convergences on the order bounded linear operators and on the orthomorphisms on a Dedekind complete vector lattice. Without further ado, let us simply state the answers and explain how they are obtained.

For a general net of order bounded linear operators \uppars{resp.\ orthomorphisms} on a general Dedekind complete vector lattice, the implications between order convergence, unbounded order convergence, convergence in a possible Hausdorff \uoLtop, strong order convergence, strong unbounded order convergence, and strong convergence with respect to a possible Hausdorff \uoLtop, are given in Table~\ref{table:convergences_general}
 \uppars{resp.\ Table~\ref{table:convergences_orthomorphisms}}.
	\begin{center}
		\begin{threeparttable}[H]
			\small
			\caption{\small Implications between convergences of nets in $\obops(E)$. }
			\begin{tabular}{ |  C{4em} |  C{4em} |  C{4em} | C{4em} |  C{4em} |  C{4em}|  C{4em} | }
				\hline
				& \o & $\uo$ & $\uoLt_{\obops(E)}$ & $\SO$ & $\SUO$ & $\STRONG\uoLt_E$ \\
				\hline
				$\o$ & $1$ & $1$ & $1$ &$1$ &$1$ &$1$ \\
				\hline
				$\uo$  & $0$ & $1$ & $1$ & $0$ & $0$ & $0$ \\
				\hline
				$\uoLt_{\obops(E)}$  & $0$ & $0$ & $1$ & $0$ & $0$ & $0$ \\
				\hline
				$\SO$  & $0$ & $0$ & $0$ & $1$ & $1$ & $1$\\
				\hline
				$\SUO$ & $0$ & $0$ & $0$ & $0$ & $1$ & $1$ \\
				\hline
				$\STRONG\uoLt_E$ & $0$ & $0$ & $0$ & $0$ & $0$ & $1$\\
				\hline
			\end{tabular}\label{table:convergences_general}
		\end{threeparttable}
	\end{center}
	\bigskip
	\begin{center}
		\begin{threeparttable}[H]
			\small
			\caption{\small Implications between convergences of nets in $\Orth(E)$.}
			\begin{tabular}{ |  C{4em} |  C{4em} |  C{4em} | C{4em} |  C{4em} |  C{4em}|  C{4em} | }
				\hline
				& $\o$ & $\uo$ & $\uoLt_{\Orth(E)}$ & $\SO$ & $\SUO$ & $\STRONG\uoLt_E$ \\
				\hline
				$\o$ & $1$ & $1$ & $1$ &$1$ &$1$ &$1$ \\
				\hline
				$\uo$ & $0$ & $1$ & $1$ & $0$ & $1$ & $1$\\
				\hline
				$\uoLt_{\Orth(E)}$  & $0$ & $0$ & $1$ & $0$ & $0$ & $1$ \\
				\hline
				$\SO$ & $0$ & $1$  & $1$  & $1$ & $1$ & $1$ \\
				\hline
				$\SUO$  & $0$ & $1$ & $1$ & $0$ & $1$ & $1$ \\
				\hline
				$\STRONG\uoLt_E$  & $0$ & $0$ & $1$ & $0$ & $0$ & $1$\\
				\hline
			\end{tabular}\label{table:convergences_orthomorphisms}
\begin{tablenotes}
\item [{}]\center{In $\Orth(E)$, \uo\ and \SUO\ convergence of nets coincide, as do a possible $\uoLt_{\Orth}(E)$ and $\STRONG\uoLt_E$ convergence.}
\end{tablenotes}
\end{threeparttable}
	\end{center}

	\bigskip
	
	In these tables, the value in a cell indicates whether the convergence of a net in the sense that labels the row of that cell does \uppars{value 1} or does not \uppars{value 0} in general imply its convergence (to the same limit) in the sense that labels the column of that cell. For example, the value 0 in the cell $(\uo,\STRONG\uoLt_E)$ in Table~\ref{table:convergences_general} indicates that there exists a net of order bounded linear operators on a Dedekind complete vector lattice $E$ that admits a Hausdorff \uoLtop\ $\uoLt_E$, such that this net is unbounded order convergent to zero in $\obops(E)$, but not strongly convergent to zero with respect to $\uoLt_E$.  The value 1 in the cell $(\uo,\STRONG\uoLt_E)$ in Table~\ref{table:convergences_orthomorphisms}, however, indicates that every net of \emph{orthomorphisms} on an arbitrary Dedekind complete vector lattice $E$ that admits a Hausdorff \uoLtop\ $\uoLt_E$, such that this net is unbounded order convergent to zero, \emph{is} strongly convergent to zero with respect to $\uoLt_E$.

	\medskip
	
	We shall now explain how these tables can be obtained.

	Obviously, the order convergence of a net of operators implies its unbounded order convergence, which implies its convergence in a possible Hausdorff \uoLtop. There are similar implications for the three associated strong convergences. Furthermore, an implication that fails for orthomorphisms also fails in the general case. Using these basic facts, it is a logical exercise to complete the tables from a few `starting values' that we now validate.

	For Table~\ref{table:convergences_general}, we have the following `starting values':
	\begin{itemize}
	\item the value 1 in the cell $(\o,\SO)$ follows from \cite[Lemma~4.1]{deng_de_jeu:2021};
	\item the value 0 in the cell $(\uo,S\uoLt_E)$ follows from \cite[Example~5.3]{deng_de_jeu:2021}, when using that, for an atomic vector lattice as in that example, the unbounded order convergence of a net and the convergence in the Hausdorff \uoLtop\ coincide (this follows from the combination of  \cite[Lemma~3.1]{dabboorasad_emelyanov_marabeh:2020} and \cite[Lemma~7.4]{taylor:2019});
	\item the value 0 in the cell $(\SO,\uoLt_{\obops(E)})$ follows from the case where $p=\infty$ in \cite[Example~5.5]{deng_de_jeu:2021}. The reason is\textemdash we resort to the notation and context of that example\textemdash that, for $p=\infty$, it follows from  \cite[Example~10.1.2]{bogachev_MEASURE_THEORY_VOLUME_II:2007} that the sequence $\mathbb E_n f$ is order bounded in $\Ell_\infty([0,1])$ for all $f\in\Ell_\infty([0,1])$. Since we already know from the general case that it is almost everywhere convergent to $f$ it is, in fact, order convergent to~$f$ in $\Ell_\infty([0,1])$. The remainder of the arguments in the example then validate the value 0 in the cell.
	\end{itemize}	

	For Table~\ref{table:convergences_orthomorphisms}, we have the following `starting values':	
	\begin{itemize}
	\item the values 0 in the cells $(\uo,\o)$, $(\uo,\SO)$, and $(\uoLt_{\Orth(E)},\uo)$, as well as in $(\uoLt_{\Orth(E)},\SUO)$, follow from the examples preceding \cite[Lemma~9.1]{deng_de_jeu:2021}, letting the multiplication operators act on the constant function 1 for the second and fourth of these cells;
	\item the value 0 in the cell $(\SO,\o)$ follows from the example following the proof of \cite[Theorem~9.4]{deng_de_jeu:2021};
	\item the values 1 in the cells $(\uo,\SUO)$ and $(\SUO,\uo)$ follow from \cite[Theorem~9.7]{deng_de_jeu:2021};
	\item the values 1 in the cells $(\uoLt_{\Orth(E)},\STRONG\uoLt_E)$ and $(\STRONG\uoLt_E, \uoLt_{\Orth(E)})$ follow from \cite[Theorem~9.10]{deng_de_jeu:2021}.
	\end{itemize}

It is easily checked that the above information suffices to complete both tables.

\setcounter{theorem}{\value{table}}

\begin{remark}\quad
\begin{enumerate}
	\item\label{part:tables_1} Every \emph{order bounded} net of orthomorphisms on an arbitrary Dedekind complete vector lattice $E$ that is strongly order convergent to zero, is order convergent to zero in $\Orth(E)$; see \cite[Theorem~9.4]{deng_de_jeu:2021};
	\item\label{part:tables_2} Every \emph{sequence} of orthomorphisms on a Dedekind complete Banach lattice $E$ that is strongly order convergent to zero, is order convergent to zero in $\Orth(E)$; see \cite[Theorem~9.5]{deng_de_jeu:2021};
	\item\label{part:tables_3} The validity of all zeroes in Table~\ref{table:convergences_general} \uppars{resp.\ Table~\ref{table:convergences_orthomorphisms}} follows from the existence of a net of order bounded linear operators \uppars{resp.\ orthomorphisms} on a Dedekind complete \emph{Banach lattice} for which the implication in question does not hold. With the cell $(\SO,\o)$ in Table~\ref{table:convergences_orthomorphisms} as the only exception, such a net of operators on a Banach lattice can even be taken to be a sequence. This follows from an inspection of the (counter) examples referred to above when validating the `starting' zeroes in the tables.
\end{enumerate}
\end{remark}

\section{Continuity of left and right multiplications}\label{sec:continuity_in_one_variable}

\noindent
In this section, we study continuity properties of left and right multiplication operators. For example, take an arbitrary $T\in\obops(E)$, where $E$ is an arbitrary Dedekind complete vector lattice that admits a Hausdorff \uoLtop\ $\uoLt_{\obops(E)}$. Is it then true that $\lambda_T:\obops(E)\to\obops(E)$ maps unbounded order convergent nets in $\obops(E)$ to  $\uoLt_{\obops(E)}$ convergent nets (with corresponding limits)? If not, is this then true when we suppose that $T\in\ocops(E)$? If not, is this true when we suppose that $T\in\Orth(E)$? One can ask a similar combination of questions, specifying to classes of increasingly well-behaved operators, for each of the $6 \cdot 6 = 36$ combinations of convergences of nets in $\obops(E)$ under consideration in this paper. There are also 36 combinations to be considered for right multiplication operators. This section provides the answers in all 72 cases; the results are contained in the Tables~\ref{table:right_multiplication} and~\ref{table:left_multiplication}. For the example that we gave, the answer is negative and remains so for arbitrary $T\in\ocops(E)$, but it becomes positive for arbitrary $T\in\Orth(E)$.

For $\Orth(E)$, there are similar questions to be asked for its left and right regular representation, but their number is smaller. Firstly, we see no obvious better-behaved subclass of $\Orth(E)$ that we should also consider. Secondly, since $\Orth(E)$ is commutative, there is only one type of multiplication involved. Thirdly, as in Table~\ref{table:convergences_orthomorphisms}, there are two pairs of coinciding convergences. All in all, there are only 4 x 4 = 16 possible combinations that actually have to be considered for the regular representation of $\Orth(E)$. Also in this case, all answers can be given; the results are contained in Table~\ref{table:multiplication_on_orthomorphisms}. As it turns out, Table~\ref{table:multiplication_on_orthomorphisms} is identical to Table~\ref{table:convergences_orthomorphisms}. There appears to be no a priori reason for this fact; it is simply the outcome.

We shall now set out to validate the Tables~\ref{table:right_multiplication},~\ref{table:left_multiplication}, and~\ref{table:multiplication_on_orthomorphisms}. Fortunately, we do not need individual results for every cell. Upon considering the multiplications by the orthomorphism that is the identity operator, the zeroes in the Tables~\ref{table:convergences_general} and~\ref{table:convergences_orthomorphisms} already determine the values in many cells. For the remaining ones, the combination of the `standard' implications that were already used for the Tables~\ref{table:convergences_general} and~\ref{table:convergences_orthomorphisms} and a limited number of results and (counter) examples already suffices. We shall now start to collect these.

We start with \o-\o\ and \SO-\SO\ continuity.

\begin{proposition}\label{res:RL_order}
Let $E$ be a Dedekind complete vector lattice.  Then:
	\begin{enumerate}
		\item\label{res:RL_order1} $\Ri{T}:\obops(E)\to\obops(E)$ is \o-\o\ continuous for all $T\in \obops(E)$;
		\item\label{res:RL_order2} $\Le{T}:\obops(E)\to\obops(E)$ is \o-\o\ continuous for all $T\in \ocops(E)$;
		\item\label{res:RL_order3} $\Ri{T}:\obops(E)\to\obops(E)$ is \SO-\SO\ continuous for all $T\in \obops(E)$;
		\item\label{res:RL_order4} $\Le{T}: \obops(E)\to\obops(E)$ is \SO-\SO\ continuous for all $T\in \ocops(E)$.
	\end{enumerate}
\end{proposition}

\begin{proof} We prove the parts~\ref{res:RL_order1} and~\ref{res:RL_order2}.
	Take $T\in \obops(E)$, and let $\opnet\subseteq \obops(E)$ be a net such that $S_\alpha\oconv 0$ in $\obops(E)$. By passing to a tail, we may assume that $(\abs{S_\alpha})_{\alpha\in\is}$ is order bounded in $\obops(E)$. Set $R_\alpha\coloneqq \bigvee_{\beta\geq \alpha} \abs{S_\beta}$ for $\alpha\in\is$. Then $\abs{S_\alpha}\leq R_\alpha$ for $\alpha\in\is$ and $R_\alpha\downarrow 0$ in $\obops(E)$ (see \cite[Remark~2.2]{gao_troitsky_xanthos:2017}). It is immediate from \cite[Theorem 1.18]{aliprantis_burkinshaw_POSITIVE_OPERATORS_SPRINGER_REPRINT:2006} that also $R_\alpha\abs{T}\downarrow 0$ in $\obops(E)$. Since $\abs{\Ri{T}(S_\alpha)}\leq R_\alpha\abs{T}$ for $\alpha\in\is$, we see that $\Ri{T}(S_\alpha)\oconv 0$ in $\obops(E)$, as desired. Suppose that, in fact, $T\in\ocops(E)$. Since $R_\alpha x\downarrow 0$ for $x\in\pos{E}$ by \cite[Theorem 1.18]{aliprantis_burkinshaw_POSITIVE_OPERATORS_SPRINGER_REPRINT:2006}, we then also have that $\abs{T}R_\alpha x\downarrow 0$ for $x\in \pos{E}$. Hence $\abs{T}R_\alpha\downarrow 0$ in $\obops(E)$. The fact that $\abs{\Le{T}(S_\alpha)}\leq \abs{T} R_\alpha$ for $\alpha\in\is$ then implies that $\Le{T}(S_\alpha)\oconv 0$ in $\obops(E)$.
	
	The parts~\ref{res:RL_order3} and~\ref{res:RL_order4} are immediate consequences of the definitions.
\end{proof}

We now show that the condition in the parts~\ref{res:RL_order2} and~\ref{res:RL_order4} of \cref{res:RL_order}  that $T\in\ocops(E)$ cannot be relaxed to $T\in\obops(E)$.

\begin{examples}\label{exam:RL_order}
	Take $E=\ell_\infty$, let $\seq{e_n}{n}$ be the sequence of standard unit vectors in $E$, and let $c$ denote the sublattice of $E$ consisting of the convergent sequences. We define a positive linear functional $f_c$ on $c$ by setting
	\[
	f_c(x)\coloneqq\lim_{n\to\infty} x_n
	\]
	for $x=\bigvee_{i=1}^\infty x_ie_i\in c$. Since $c$ is a majorising vector subspace of $E$, \cite[Theorem~1.32]{aliprantis_burkinshaw_POSITIVE_OPERATORS_SPRINGER_REPRINT:2006} shows that there exists a positive functional $f$ on $E$ that extends $f_c$. We define $T:E\to E$ by setting $Tx=f(x)e_1$ for $x\in E$. Clearly, $T\in \obops(E)$; a consideration of $T(\bigvee_{i=n}^\infty e_i)$ for $n\geq 1$ shows that $T\notin\ocops(E)$.
			
	We define $S_{n}\in \ocops(E)$ for $n\geq 1$ by setting
	\[
	S_{n}x\coloneqq x_1\bigvee_{i=1}^{n} e_i,
	\]
	and $S\in \ocops(E)$ by setting
	\[
	Sx\coloneqq x_1 \bigvee_{i=1}^\infty e_i
	\]
	for $x=\bigvee_{i=1}^\infty x_ie_i\in E$. Clearly, $S_n\uparrow S$ in $\obops(E)$. On the other hand, $\Le{T}(S_n)=0$ for all $n\geq 1$, while $\Le{T}(S)=P_1$, where $P_1\in\obops(E)$ is the order projection onto the span of $e_1$. This shows that $\Le{T}:\obops(E)\to\obops(E)$ is not \o-\o\ continuous.
	
	The sequence $\seq{S_n}{n}$, being order convergent to $S$, is also strongly un\-bound\-ed order convergent to $S$ in $\obops(E)$. Hence $\Le{T}:\obops(E)\to\obops(E)$ is not \SO-\SO\ continuous.
\end{examples}

\begin{remark}\label{rem:RL_order}
	\cref{exam:RL_order} also shows that, already for a Banach lattice $E$, $\Le{T}$ need not even be sequentially \o-$\uoLt_{\obops(E)}$ continuous, sequentially \o-\STRONG$\uoLt_E$ continuous, sequentially \SO-$\uoLt_{\obops(E)}$ continuous, or sequentially \SO-\STRONG$\uoLt_E$ continuous for arbitrary $T\in\obops(E)$.
\end{remark}

\begin{remark}
	The \o-\o\ continuity (appropriately defined) of left and right multiplications on ordered algebras is studied in \cite{alekhno:2017}. It is established on \cite[p.~542--543]{alekhno:2017} that, for a Dedekind complete vector lattice $E$, the right and left multiplication by an element $T$ of the ordered algebra $L(E)$ of all (!) linear operators on $E$ are both order continuous on $L(E)$ in the sense of \cite{alekhno:2017} if and only if the left multiplication is, which is the case if and only if $T\in \ocops(E)$. The proof refers to  \cite[Example~2.9 (a)]{alekhno:2012}, which is concerned with multiplications by a positive operator $T$ on the ordered Banach algebra $L(E)$ of all (!) bounded linear operators on a Dedekind complete Banach lattice $E$. It is established in that example that the simultaneous order continuity of the right and left multiplication by $T$ on $L(E)$ in the sense of \cite{alekhno:2012} is equivalent to $T$ being order continuous. On \cite[p.~151]{alekhno:2012} it is mentioned that this criterion for the order continuity of an operator can also be presented for an arbitrary Dedekind complete vector lattice. Although it is not stated as such, and although a proof as such is not given, the author may have meant to state, and have known to be true, that, for a Dedekind complete vector lattice $E$ and $T\in\obops(E)$, $\Le{T}$ and $\Ri{T}$ are both \o-\o\ continuous on $\obops(E)$ in the sense of the present paper if and only if $\Le{T}$ is, which is the case if and only if $T\in\ocops(E)$. Using arguments as on \cite[p.~151]{alekhno:2012} and \cite[p.~542--543]{alekhno:2017}, the authors of the present paper have verified that\textemdash this is the hard part\textemdash for $T\in\obops(E)$, the \o-\o\ continuity of $\Le{T}$ on $\obops(E)$ in the sense of the present paper does imply that $T\in\ocops(E)$. Hence the three properties of $T\in\obops(E)$ mentioned above are, indeed, equivalent; a result that is to be attributed to the late Egor Alekhno.
\end{remark}

We use the opportunity to establish the following side result, which follows easily from combining each of \cite[Satz~3.1]{synnatzschke:1980} and \cite[Proposition~2.2]{chen_schep:2016} with the parts~\ref{res:RL_order1} and~\ref{res:RL_order2} of \cref{res:RL_order}.
		
\begin{proposition}\label{res:regular_representations}
Let $E$ be a Dedekind complete vector lattice. Then:
\begin{enumerate}
	\item\label{part:regular_representations_1} the map $T\mapsto\Ri{T}$ defines an order continuous lattice homomorphism $\Ri{}:\obops(E)\to\ocops(\ocops(E),\obops(E))$.
	\item\label{part:regular_representations_2} the map $T\mapsto\Le{T}$ defines an order continuous lattice homomorphism $\Le{}:\obops(E)\to\obops(\obops(E))$ that maps $\ocops(E)$ into $\ocops(\obops(E))$.
\end{enumerate}		
\end{proposition}

\begin{remark}
	In \cite[Problem~1]{wickstead:2017c}, it was asked, among others, whether, for a Dedekind complete vector lattice $E$, the left regular representation of $\obops(E)$ is a lattice homomorphism from $\obops(E)$ into $\obops(\obops(E))$. In \cite[Theorem~11.19]{dales_de_jeu:2019}, it was observed that the affirmative answer is, in fact, provided by \cite[Satz~3.1]{synnatzschke:1980}. Part~\ref{part:regular_representations_2} of \cref{res:regular_representations} gives still more precise information.
	
	 Part~\ref{part:regular_representations_1}, which relies on \cite[Proposition~2.2]{chen_schep:2016}, 		
	implies that the right regular representation of $\ocops(E)$ is an order continuous lattice homomorphism from $\ocops(E)$ into $\obops(\ocops(E))$, with an image that is, in fact, contained in $\ocops(\ocops(E))$.
\end{remark}

After this brief digression, we continue with the main line of this section, and consider \uo-\uo\ and \SUO-\SUO\ continuity of left and right multiplication operators.

\begin{proposition}\label{res:RL_unbounded_order}
Let $E$ be a Dedekind complete vector lattice.  Then:
\begin{enumerate}
    \item\label{res:RL_unbounded_order1} $\Ri{T}$ is \uo-\uo\ continuous on  $\obops(E)$ for all $T\in \Orth(E)$;
    \item\label{res:RL_unbounded_order2} $\Le{T}$ is \uo-\uo\ continuous on $\obops(E)$ for all $T\in \Orth(E)$;
     \item\label{res:RL_unbounded_order3} $\Ri{T}$ is \SUO-\SUO\ continuous on  $\obops(E)$ for all $T\in \obops(E)$;
     \item\label{res:RL_unbounded_order4} $\Le{T}$ is \SUO-\SUO\ continuous on  $\obops(E)$ for all $T\in \Orth(E)$.
\end{enumerate}
\end{proposition}

\begin{proof}

For the parts~\ref{res:RL_unbounded_order1} and~\ref{res:RL_unbounded_order2}, we note that \cref{res:orthisorth} shows that $\Ri{T},\Le{T}\in\Orth(\obops(E))$. Their \uo-\uo\ continuity then follows from \cite[Proposition~7.1]{deng_de_jeu:2021}

Part~\ref{res:RL_unbounded_order3} is trivial.

We prove part~\ref{res:RL_unbounded_order4}. Let $\opnet$  be a net in $\obops(E)$ such that $S_\alpha\conv{\SUO}0$. Take an $x\in E$. Then $S_\alpha x\uoconv 0$ in $E$. It follows from   \cite[Proposition~7.1]{deng_de_jeu:2021} that $\Le{T}(S_\alpha)x=TS_\alpha x\uoconv 0$ in $E$, as desired.
\end{proof}

\begin{remark}
	For the proof of the parts~\ref{res:RL_unbounded_order1} and~\ref{res:RL_unbounded_order2} of \cref{res:RL_unbounded_order}, an appeal to the beginning of \cite[Section~2]{luxemburg_de_pagter:2002} can replace the use of \cref{res:orthisorth}. It is, however, only \cref{res:orthisorth} that permits the obvious extensions of the parts~\ref{res:RL_unbounded_order1} and~\ref{res:RL_unbounded_order2} of \cref{res:RL_unbounded_order} to (not necessarily regular) vector sublattices of $\obops(E)$ that are invariant under left or right composition with orthomorphisms, provided that they have the principal projection property.
\end{remark}

We now show that the condition in the parts~\ref{res:RL_unbounded_order1}, \ref{res:RL_unbounded_order2}, and~\ref{res:RL_unbounded_order4} of \cref{res:RL_unbounded_order} that $T\in\Orth(E)$ cannot be relaxed to $T\in\ocops(E)$.

\begin{examples}\label{exam:RL_unbounded_order}\quad
\begin{enumerate}
\item \label{exam:RL_unbounded_order1} We first give an example showing that $\Le{T}$ and $\Ri{T}$ need not be \uo-\uo\ continuous on $\obops(E)$ for $T\in \ocops(E)$.

Let $E=\Ell_p([0,1])$ with $1\leq p<\infty$. We define $T\in\obops(E)=\ocops(E)$ by setting
\begin{equation}\label{eq:RL_unbounded_order}
Tf\coloneqq \int f\,\mathrm{d}\mu \cdot \chi_{[0,1]}
\end{equation}
for $f\in E$. For $n\geq 1$, we define the positive operator $S_n$ on $E$ by setting
\begin{equation*}
S_nf(t)\coloneqq
\begin{cases}
f(t+1/n) & \text{ for }{t\in [0,(n-1)/n)};\\
f(t-(n-1)/n) & \text{ for } {t\in [(n-1)/n,1]}.
\end{cases}
\end{equation*}
We claim that $(S_n)_{n=1}^\infty$ is a disjoint sequence in $\obops(E)$. Let $m,n\geq 1$ with $m>n$. Take a $k\geq 1$ such that $1/k< 1/n-1/m$. For every $f\in E^+$,  \cite[Theorem~1.51]{aliprantis_burkinshaw_POSITIVE_OPERATORS_SPRINGER_REPRINT:2006} then implies that
  \[
  0\leq S_m\wedge S_n (f)\leq \sum_{i=1}^k S_m(f\cdot \chi_{[(i-1)/k,i/k]})\wedge S_n(f\cdot \chi_{[(i-1)/k,i/k]})=0
  \]
  because the supports of $S_m(f\cdot \chi_{[(i-1)/k,i/k]})$ and $S_n(f\cdot \chi_{[(i-1)/k,i/k]})$ are disjoint for $i=1,\dotsc, k$. Hence $S_m\wedge S_n=0$, as claimed.

By \cite[Corollary~3.6]{gao_troitsky_xanthos:2017}, the disjoint sequence $(S_n)_{n=1}^\infty$ is unbounded order convergent to zero in $\obops(E)$. On the other hand, it is easy to see that $\Ri{T}(S_n)=\Le{T}(S_n)=T\neq 0$ for all $n\geq 1$. Hence neither of $\seq{\Ri{T}(S_n)}{n}$ and $\seq{\Le{T}(S_n)}{n}$ is unbounded order convergent to zero in $\obops(E)$. This shows that neither $\Ri{T}$ nor $\Le{T}$ is \uo-\uo\ continuous on $\obops(E)$.

\item\label{exam:RL_unbounded_order2} We now give an example showing that $\Le{T}$ need not be \SUO-\SUO\ continuous on $\obops(E)$ for $T\in \ocops(E)$. Let $E=\Ell_p([0,1])$ with $1\leq p<\infty$. For $n\geq 1$, define the positive operator $S_n$ on $E$ by setting $S_n f\coloneqq 2^n \chi_{[1-1/2^{n-1},1-1/2^n]}\cdot f$ for $f\in E$.
Let $T\in\ocops(E)$ be defined as in \cref{eq:RL_unbounded_order}. For every $f\in E$, it is clear that $S_n f$ and $S_m f$ are disjoint whenever $m\neq n$, and then \cite[Corollary~3.6]{gao_troitsky_xanthos:2017} shows that $S_nf\uoconv 0$ in $E$. That is,  $\seq{S_n}{n}$ is strongly unbounded order convergent to zero. On the other hand, it is easily seen that $\Le{T}(S_n)\chi_{[0,1]}=\chi_{[0,1]}\neq 0$ for $n\geq 1$. This implies that $\seq{\Le{T}(S_n)}{n}$ is not strongly unbounded order convergent to zero, so that $\Le{T}$ is not \SUO-\SUO\ continuous on $\obops(E)$.
\end{enumerate}
\end{examples}

\begin{remark}\label{rem:RL_unbounded_order}
	\cref{exam:RL_unbounded_order} also shows that $\Ri{T}$ and $\Le{T}$ need not be \uo-$\uoLt_{\obops(E)}$\ continuous and that $\Le{T}$ need not be \SUO-\STRONG$\uoLt_E$ continuous on $\obops(E)$ for arbitrary $T\in\ocops(E)=\obops(E)$. In fact, the sequential versions of these continuity properties can already fail to hold, even in cases where $E$ is a Banach lattice with an order continuous norm.	
\end{remark}

We now turn to the Hausdorff \uoLtops, where we shall make use of \cref{res:all_three_or_none}.

\begin{proposition}\label{res:RL_uoLt}
Let $E$ be a Dedekind complete vector lattice that admits a \necun\ Hausdorff uo-Lebesgue topology $\uoLt_E$, so that $\obops(E)$ also admits a \necun\ Hausdorff uo-Lebesgue topology $\uoLt_{\obops(E)}$. Then:
\begin{enumerate}
    \item\label{res:RL_uoLt1} $\Ri{T}$ is $\uoLt_{\obops(E)}$-$\uoLt_{\obops(E)}$ continuous on  $\obops(E)$ for all $T\in \Orth(E)$;
    \item\label{res:RL_uoLt2} $\Le{T}$ is $\uoLt_{\obops(E)}$-$\uoLt_{\obops(E)}$ continuous on $\obops(E)$ for all $T\in \Orth(E)$;
     \item\label{res:RL_uoLt3} $\Ri{T}$ is \STRONG$\uoLt_E$-\STRONG$\uoLt_E$ continuous on  $\obops(E)$ for all $T\in \obops(E)$;
     \item\label{res:RL_uoLt4} $\Le{T}$ is \STRONG$\uoLt_E$-\STRONG$\uoLt_E$ continuous on  $\obops(E)$ for all $T\in \Orth(E)$.
\end{enumerate}
\end{proposition}

\begin{proof}
We know from \cref{res:orthisorth} that $\Ri{T},\Le{T}\in \Orth(\obops(E))$ when $T\in\Orth(E)$, and then the parts~\ref{res:RL_uoLt1} and~\ref{res:RL_uoLt2} follow from \cite[Corollary~7.3]{deng_de_jeu:2021}.

Part~\ref{res:RL_uoLt3} is trivial.

Part~\ref{res:RL_uoLt4} follows from \cite[Corollary~7.3]{deng_de_jeu:2021}.
\end{proof}

We now show that the condition  in the parts~\ref{res:RL_uoLt1}, \ref{res:RL_uoLt2}, and~\ref{res:RL_uoLt4} of \cref{res:RL_uoLt} that $T\in\Orth(E)$ cannot be relaxed to $T\in\ocops(E)$.

\begin{examples}\label{exam:RL_uoLt}\quad
	\begin{enumerate}
		\item  We first give an example showing that $\Le{T}$ and $\Ri{T}$ need not be  $\uoLt_{\obops(E)}$-$\uoLt_{\obops(E)}$ continuous on $\obops(E)$ for $T\in \ocops(E)$. For this, we resort to the context and notation of part~\ref{res:RL_unbounded_order1} of \cref{exam:RL_unbounded_order}. In that example, we know that $S_n\uoconv 0$ in $\obops(E)$, and then certainly $S_n\conv{\uoLt_{\obops(E)}} 0$.  Since $\Ri{T}(S_n)=\Le{T}(S_n)=T\neq 0$ for all $n\geq 1$, we see that neither $\Ri{S}$ nor $\Le{S}$ is $\uoLt_{\obops(E)}$-$\uoLt_{\obops(E)}$ continuous on  $\obops(E)$.
		\item We give an example showing that $\Le{T}$ need not be \STRONG$\uoLt_E$-\STRONG$\uoLt_E$ continuous on  $\obops(E)$ for $T\in\ocops(E)$. For this, we resort to the context and notation of part~\ref{res:RL_unbounded_order2} of \cref{exam:RL_unbounded_order}. In that example, we know that $S_nf\uoconv 0$ in $E$ for $f\in E$. Then certainly $S_n f\conv{\uoLt_E} 0$ for $f\in E$. Since $\Le{T}(S_n)\chi_{[0,1]}=\chi_{[0,1]}\neq 0$ for all $n\geq 1$, we see that $\Le{T}$ is not \STRONG$\uoLt_E$-\STRONG$\uoLt_E$ continuous.
	\end{enumerate}
\end{examples}

\begin{remark}\label{rem:RL_uoLt}
	\cref{exam:RL_uoLt} also shows that $\Ri{T}$ and $\Le{T}$ need not even be $\uoLt_{\obops(E)}$-$\uoLt_{\obops(E)}$ continuous and that $\Le{T}$ need not even be $\STRONG\uoLt_{\obops(E)}$-\STRONG$\uoLt_{\obops(E)}$ continuous on $\obops(E)$ for arbitrary $T\in\ocops(E)=\obops(E)$. In fact, the sequential versions of these continuity properties can already fail to hold, even in cases where $E$ is a Banach lattice with an order continuous norm.	
\end{remark}

We now have sufficient material at our disposal to determine the tables mentioned at the beginning of this section.

For right multiplications on $\obops(E)$, the results are in Table~\ref{table:right_multiplication}. The value in a cell with a row label indicating a convergence structure $\convstru_1$ and a column label indicating a convergence structure $\convstru_2$ is to be interpreted as follows:
\begin{enumerate}
	\item A value $\{0\}$ (resp.\ $\Orth(E)$, resp.\ $\ocops(E)$) means that $\Ri{T}$ is $\convstru_1$-$\convstru_2$ continuous on $\obops(E)$ for every Dedekind complete vector lattice $E$ and for every $T\in\{0\}$ (resp.\ $T\in\Orth(E)$, resp.\ $T\in\ocops(E)$), but there exist a Dedekind complete vector lattice $E$ and a $T\in\Orth(E)$ (resp.\ $T\in\ocops(E)$, resp.\ $T\in\obops(E)$) for which this is  not the case;
	\item A value $\obops(E)$ means that $\Ri{T}$ is $\convstru_1$-$\convstru_2$ continuous on $\obops(E)$ for every Dedekind complete vector lattice $E$ and for every $T\in\obops(E)$.
\end{enumerate}	

\setcounter{table}{\value{theorem}}

	\begin{center}
		\begin{threeparttable}[H]
			\small
			\caption{\small Continuity of right multiplications on $\obops(E)$.}
			\begin{tabular}{ |  C{4em} |  C{4em} |  C{4em} | C{4em} |  C{4em} |  C{4em}|  C{4em} | }
				\hline
				& $\o$ & $\uo$ & $\uoLt_{\obops(E)}$ & $\SO$ & $\SUO$ & $\STRONG\uoLt_E$ \\
				\hline
				$\o$ & $\obops(E)$ & $\obops(E)$ & $\obops(E)$ & $\obops(E)$ &$\obops(E)$ & $\obops(E)$ \\
				\hline
				$\uo$  & $\{0\}$ & $\Orth(E)$ & $\Orth(E)$ & $\{0\}$ & $\{0\}$ & $\{0\}$ \\
				\hline
				$\uoLt_{\obops(E)}$  & $\{0\}$ & $\{0\}$ & $\Orth(E)$ & $\{0\}$ & $\{0\}$ & $\{0\}$ \\
				\hline
				$\SO$  & $\{0\}$ & $\{0\}$ & $\{0\}$ & $\obops(E)$ & $\obops(E)$ & $\obops(E)$\\
				\hline
				$\SUO$ & $\{0\}$ & $\{0\}$ & $\{0\}$ & $\{0\}$ & $\obops(E)$ & $\obops(E)$ \\
				\hline
				$\mathrm{S}\uoLt_E$ & $\{0\}$ & $\{0\}$ & $\{0\}$ & $\{0\}$ & $\{0\}$ & $\obops(E)$\\
				\hline
			\end{tabular}\label{table:right_multiplication}
		\end{threeparttable}
	\end{center}

\vskip 1em

As mentioned in the beginning of this section, a zero in Table~\ref{table:convergences_general} gives $\{0\}$ in Table~\ref{table:right_multiplication}. It is easily verified that the remaining values can be determined using that order convergence implies unbounded order convergence, which implies $\uoLt_E$ convergence when applicable; that analogous implications hold for their strong versions; that order convergence implies strong order convergence; combined with \cref{res:RL_order}, \cref{res:RL_unbounded_order}, \cref{rem:RL_unbounded_order}, \cref{res:RL_uoLt}, and \cref{rem:RL_uoLt}.

For left multiplications on $\obops(E)$, the results are in Table~\ref{table:left_multiplication}, with a similar interpretation of the values in the cells as for Table~\ref{table:right_multiplication}.

	\begin{center}
		\begin{threeparttable}[H]
			\small
			\caption{\small Continuity of left multiplications on $\obops(E)$.}
			\begin{tabular}{ |  C{4em} |  C{4em} |  C{4em} | C{4em} |  C{4em} |  C{4em}|  C{4em} | }
				\hline
				& $\o$ & $\uo$ & $\uoLt_{\obops(E)}$ & $\SO$ & $\SUO$ & $\mathrm{S}\uoLt_E$ \\
				\hline
				$\o$ & $\ocops(E)$ & $\ocops(E)$ & $\ocops(E)$ &$\ocops(E)$ &$\ocops(E)$ &$\ocops(E)$ \\
				\hline
				$\uo$ & $\{0\}$ & $\Orth(E)$ & $\Orth(E)$ & $\{0\}$ & $\{0\}$ & $\{0\}$\\
				\hline
				$\uoLt_{\obops(E)}$  & $\{0\}$ & $\{0\}$ & $\Orth(E)$ & $\{0\}$ & $\{0\}$ & $\{0\}$ \\
				\hline
				$\SO$ & $\{0\}$ & $\{0\}$  & $\{0\}$  & $\ocops(E)$ & $\ocops(E)$ & $\ocops(E)$ \\
				\hline
				$\SUO$  & $\{0\}$ & $\{0\}$ & $\{0\}$ & $\{0\}$ & $\Orth(E)$ & $\Orth(E)$ \\
				\hline
				$\STRONG\uoLt_E$  & $\{0\}$ & $\{0\}$ & $\{0\}$ & $\{0\}$ & $\{0\}$ & $\Orth(E)$\\
				\hline
			\end{tabular}\label{table:left_multiplication}
		\end{threeparttable}
	\end{center}

\vskip 1em
 For Table~\ref{table:left_multiplication}, the values of the cells can be determined using the zeroes in Table~\ref{table:convergences_general}, the `standard implications' as listed for Table~\ref{table:right_multiplication}, combined with \cref{res:RL_order}, \cref{rem:RL_order}, \cref{res:RL_unbounded_order}, \cref{rem:RL_unbounded_order},  \cref{res:RL_uoLt}, and \cref{rem:RL_uoLt}.

For multiplications on $\Orth(E)$, the continuity properties are given by Table~\ref{table:multiplication_on_orthomorphisms}. In that table, a value $1$ in a cell with a row label indicating a convergence structure $\convstru_1$ and a column label indicating a convergence structure $\convstru_2$ means that the maps $\Ri{T}=\Le{T}:\Orth(E)\to\Orth(E)$ is $\convstru_1$-$\convstru_2$ continuous for all $T\in\Orth(E)$. A value $0$ means that there exists a Dedekind complete vector lattice $E$ and a $T\in\Orth(E)$ for which this is not the case.

\begin{threeparttable}[H]
	\small
	\caption{\small Continuity of multiplications on $\Orth(E)$.}
	\begin{tabular}{ |  C{4em} |  C{4em} |  C{4em} | C{4em} |  C{4em} |  C{4em}|  C{4em} | }
		\hline
		& $\o$ & $\uo$ & $\uoLt_{\Orth}(E)$ & $\SO$ & $\SUO$ & $\STRONG\uoLt_E$ \\
		\hline
		$\o$ & $1$ & $1$ & $1$ &$1$ &$1$ &$1$ \\
		\hline
		$\uo$ & $0$ & $1$ & $1$ & $0$ & $1$ & $1$\\
		\hline
		$\uoLt_{\Orth(E)}$  & $0$ & $0$ & $1$ & $0$ & $0$ & $1$ \\
		\hline
		$\SO$ & $0$ & $1$  & $1$  & $1$ & $1$ & $1$ \\
		\hline
		$\SUO$  & $0$ & $1$ & $1$ & $0$ & $1$ & $1$ \\
		\hline
		$\STRONG\uoLt_E$  & $0$ & $0$ & $1$ & $0$ & $0$ & $1$\\
		\hline
	\end{tabular}\label{table:multiplication_on_orthomorphisms}
\begin{tablenotes}
	\item [{}] \center{In $\Orth(E)$, \uo\ and \SUO\ convergence of nets coincide, as do a possible $\uoLt_{\Orth}(E)$ and $\STRONG\uoLt_E$ convergence.}
\end{tablenotes}
\end{threeparttable}

\vskip 1em
  The values in the cells of Table~\ref{table:multiplication_on_orthomorphisms} can be determined using the zeroes in Table~\ref{table:convergences_orthomorphisms}, the `standard implications' as listed for Table~\ref{table:right_multiplication}; the fact that $\Orth(E)$ is a regular vector sublattice of $\obops(E)$; the facts that unbounded order convergence and strong unbounded order convergence coincide on $\Orth(E)$, as do a possible $\uoLt_{\Orth(E)}$ and $\STRONG\uoLt_E$ convergence; combined with \cref{res:RL_order}, \cref{res:RL_unbounded_order}, and \cref{res:RL_uoLt}.

\section{Simultaneous continuity of multiplications and adherences of subalgebras of  $\obops(E)$ }\label{sec:simultaneous_continuity}

\noindent
In this section, we study the simultaneous continuity of the multiplications in subalgebras of $\obops(E)$ (where $E$ is a Dedekind complete vector lattice) with respect to the six convergence structures under consideration in this paper. This is motivated by questions of the following type. Suppose that $E$ admits a Hausdorff \uoLtop. Take a subalgebra (not necessarily a vector lattice subalgebra) $\opalg$ of $\obops(E)$. Is its adherence $\adh{\SuoLt_E}{\opalg}$ in $\obops(E)$ with respect to strong $\uoLt_E$ convergence again a subalgebra of $\obops(E)$? This is not always the case, not even when $\opalg\subseteq{\ocops(E)}$; see \cref{exam:uoLt_not_algebra}. When $\opalg\subseteq\Orth(E)$, however, the answer is affirmative; see \cref{res:uoLt_is_alg}.

As is easily verified, it follows already from the continuity of the left and right multiplications with respect to strong $\uoLt_E$ convergence (see \cref{res:RL_uoLt}) that $\adh{\SuoLt_E}{\opalg}\cdot \adh{\SuoLt_E}{\opalg}\subseteq \adh{\SuoLt_E}{ \adh{\SuoLt_E}{\opalg}}$ when $\opalg\subseteq \Orth(E)$, but that is not sufficient to show that $\adh{\SuoLt_E}{\opalg}$ is a subalgebra. The simultaneous continuity of the multiplication in $\Orth(E)$ with respect to strong $\uoLt_E$ convergence in $\Orth(E)$ would be sufficient to conclude this, and this can indeed be established; see \cref{res:simultaneous_uo-Lebesgue_continuity}.

For each of the remaining five convergence structures, we follow the same pattern. We establish (this also relies on the single variable results in \cref{sec:continuity_in_one_variable}) the simultaneous continuity of the multiplication with respect to the convergence structure under consideration, and then conclude that the pertinent adherence of a subalgebra is again a subalgebra. For the latter result it is\textemdash as the above example already indicates\textemdash essential to impose an extra condition on the subalgebra $\opalg$. This condition depends on the convergence structure under consideration. Natural extra conditions are that $\opalg$ be a subalgebra of $\Orth(E)$ or of $\ocops(E)$, and  we do indeed obtain positive results under such conditions. We also have fairly complete results showing that the relaxation of the pertinent condition to the `natural' next lenient one does, in fact, render the statement that the adherence is a subalgebra again invalid. This also implies that multiplication is then not simultaneously continuous.

In the cases where the lattice operations are known to be simultaneously continuous with respect to the convergence structure under consideration, it obviously also follows that the pertinent adherence of a vector lattice subalgebra is a vector lattice subalgebra again.

We shall now embark on this programme. We start with order convergence, which is the easiest case. For this, we have the following result on the simultaneous continuity of multiplication.

\begin{proposition}\label{res:simultaneous_order_continuity}
\!Let $E$ be a Dedekind complete vector \hskip -0.5pt lattice.\! Suppose that $\opnet$ is a net in $\ocops(E)$ such that $S_\alpha\oconv S$ in $\obops(E)$ for some $S\in\obops(E)$ and that $\opnettwo\subseteq\obops(E)$ is a net such that $T_\beta\oconv T$ in $\obops(E)$ for some $T\in\obops(E)$. Then $S\in\ocops(E)$ and $S_\alpha T_\beta\oconv ST$ in $\obops(E)$.
\end{proposition}

\begin{proof}
	It is clear that $S\in\ocops(E)$. By passing to a tail, we may suppose that $\netgen{\abs{T_\beta}}{\beta\in\istwo}$ is bounded above by some $R\in\pos{\obops(E)}$. Using the parts~\ref{res:RL_order1} and~\ref{res:RL_order2} of \cref{res:RL_order} for the final order convergence, we have that
	\begin{align*}
		\abs{S_\alpha T_\beta-ST}&\leq \abs{S_\alpha T_\beta-ST_\beta}+\abs{ST_\beta-ST}
		\\ &\leq \abs{S_\alpha-S}R+\abs{S}\abs{T_\beta-T}\oconv 0
	\end{align*}
in $\obops(E)$. Hence $S_\alpha T_\beta\oconv ST$ in $\obops(E)$.
\end{proof}

The following is now clear from \cref{res:simultaneous_order_continuity} and the simultaneous order continuity of the lattice operations.

\begin{corollary}\label{res:order_is_al}
Let $E$ be a Dedekind complete vector lattice. Suppose that $\opalg$ is a subalgebra of $\ocops(E)$. Then the adherence $\oadh{\opalg}$ in $\obops(E)$ is also a subalgebra of $\ocops(E)$. When $\opalg$ is a vector lattice subalgebra of $\ocops(E)$, then so is $\oadh{\opalg}$.
\end{corollary}

We now show that the condition in \cref{res:order_is_al}  that $\opalg\subseteq\ocops(E)$ cannot be relaxed to $\opalg\subseteq \obops(E)$.

\begin{example}\label{exam:order_is_al}
Take $E=\ell_\infty$ and let $\seq{e_n}{n}$ be the standard sequence of unit vectors in $E$. We define $T\in\obops(E)$ as in \cref{exam:RL_order}. For $n\geq 1$, we now define $S_n^\prime\in\ocops(E)$ by setting
\[
S_n^\prime x\coloneqq x_2\bigvee_{i=3}^{n+2} e_i,
\]
and $S^\prime\in \ocops(E)$ by setting
\[
S^\prime x\coloneqq x_2 \bigvee_{i=3}^\infty e_i
\]
for $x=\bigvee_{i=1}^\infty x_ie_i\in E$. It is easily verified that $T^2=0$, that $S_n^\prime S_m^\prime=0$ for $m,n\geq 1$, and that $S_n^\prime T=T S_n^\prime=0$ for $n\geq 1$. Hence $\opalg\coloneqq \Span\{T,\,S_n^\prime:n\geq 1 \}$ is a subalgebra of $\obops(E)$. As $S_n^\prime\uparrow S^\prime$ in $\obops(E)$, both $S^\prime$ and $T$ are elements of $\oadh{\opalg}$.

However, $TS^\prime\notin\oadh{\opalg}$. In fact, $TS^\prime$ is not even an element of $\Soadh{\opalg}\supseteq \oadh{\opalg}$. To see this, we observe that $TS^\prime e_2=e_1\neq 0$, and that, as is easily verified, $TS^\prime e_2\perp Re_2$ for all $R\in\opalg$. Hence there cannot exist a net $\netgen{R_\alpha}{\alpha\in\is}\subseteq\opalg$ such that $R_\alpha e_2\oconv TS^\prime e_2$ in $E$, let alone such that $R_\alpha\SOconv TS^\prime$ in $\obops(E)$.
\end{example}

Now we turn to the strong order adherences of subalgebras of $\obops(E)$. We start by showing that $\Orth(E)$ is closed in $\obops(E)$ under the convergences under consideration in this paper. We recall from \cref{res:all_three_or_none} that either all of $E$, $\Orth(E)$, and $\obops(E)$ admit a Hausdorff \uoLtop, or none does.

\begin{lemma}\label{res:orth_is_closed}
Let $E$ be a Dedekind complete vector lattice. Then $\Orth(E)$ is closed in $\obops(E)$ under order convergence, unbounded order convergence, strong order convergence, and strong unbounded order convergence. Suppose that $E$ admits a \necun\ Hausdorff \uoLtop. Then $\Orth(E)$ is closed in $\obops(E)$ under $\uoLt_{\obops(E)}$ convergence and strong $\uoLt_E$ convergence.
\end{lemma}

\begin{proof} A band in a (not necessarily Dedekind complete) vector lattice is not only closed under order convergence, but also closed under unbounded order convergence (see \cite[Proposition~3.15]{gao_troitsky_xanthos:2017}) and under convergence in a Hausdorff locally solid linear topology on the lattice (see \cite[Theorem~2.21(d)]{aliprantis_burkinshaw_LOCALLY_SOLID_RIESZ_SPACES_WITH_APPLICATIONS_TO_ECONOMICS_SECOND_EDITION:2003}. This implies that $\Orth(E)$ is closed in $\obops(E)$ under order convergence, unbounded order convergence, and convergence in a possible Hausdorff \uoLtop\ on $\obops(E)$. It also implies that, for each of the three strong convergences in $\obops(E)$ under consideration, a limit in $\obops(E)$ of a net of orthomorphisms, i.e., of order bounded band preserving operators, is again an order bounded band preserving operator, i.e., an orthomorphism.
\end{proof}

\begin{proposition}\label{res:simultaneous_strong_order_continuity}
	\!Let $E$ be a Dedekind complete vector \hskip -0.5pt lattice.\! Suppose that $\opnet$ is a net in $\Orth(E)$ such that $S_\alpha\SOconv S$ in $\obops(E)$ for some $S\in\obops(E)$ and that $\opnettwo\subseteq\obops(E)$ is a net such that $T_\beta\SOconv T$ in $\obops(E)$ for some $T\in\obops(E)$. Then $S\in\Orth(E)$, and $S_\alpha T_\beta\SOconv ST$ in $\obops(E)$.
\end{proposition}

\begin{proof}
	\cref{res:orth_is_closed} shows that $S\in\Orth(E)$. Take $x\in E$. By passing to a tail, we may suppose that $\netgen{\abs{T_\beta x}}{\beta\in\istwo}$ is bounded above by some $y\in\pos{E}$. By applying \cite[Theorem~2.43]{aliprantis_burkinshaw_POSITIVE_OPERATORS_SPRINGER_REPRINT:2006} and the order continuity of $\abs{S}$ for the final convergence, we see that
	\begin{align*}
		\abs{S_\alpha T_\beta x-ST x}&\leq \abs{(S_\alpha-S)T_\beta x}+\abs{S(T_\beta-T)x}
		\\&\leq \abs{S_\alpha-S}\abs{T_\beta x}+\abs{S}\abs{(T_\beta-T)x}
		\\&\leq \abs{S_\alpha-S}y+\abs{S}\abs{T_\beta x-Tx}
		\\&=\abs{(S_\alpha-S)y}+\abs{S}\abs{T_\beta x-Tx}\oconv 0
	\end{align*}
	in $E$. Hence $S_\alpha T_\beta\SOconv ST$ in $\obops(E)$.
\end{proof}

The following is now clear from \cref{res:simultaneous_strong_order_continuity}.

\begin{corollary}\label{res:Soad_is_alg_VL}	Let $E$ be a Dedekind complete vector lattice. Suppose that $\opalg$ is a subalgebra of $\Orth(E)$. Then the adherence $\Soadh{\opalg}$ in $\obops(E)$ is also a subalgebra of $\Orth(E)$.
\end{corollary}

We now show that the condition in \cref{res:Soad_is_alg_VL} that $\opalg\subseteq\Orth(E)$ cannot be relaxed to $\opalg\subseteq \obops(E)$. At the time of writing, the authors do not know whether it might be relaxed to $\opalg\subseteq \ocops(E)$.

\begin{example}\label{rem:Soad_is_alg}
We resort to the context and notation of \cref{exam:order_is_al}. In that example, we had operators $T,S^\prime\in\oadh{\opalg}$ such that $TS^\prime\notin\Soadh{\opalg}$. Since $\oadh{\opalg}\subseteq\Soadh{\opalg}$, this example also provides an example as currently needed.
\end{example}

We turn to unbounded order adherences and strong unbounded order adherences.

\begin{proposition}\label{res:simultaneous_unbounded_order_continuity}
	\!Let $E$ be a Dedekind complete vector \hskip -0.5pt lattice.\! Suppose that $\opnet$ is a net in $\Orth(E)$ such that $S_\alpha\conv{\uo}S$ in $\obops(E)$ for some $S\in\obops(E)$ and that $\opnettwo\subseteq\Orth(E)$ is a net such that $T_\beta\uoconv T$ in $\obops(E)$ for some $T\in\obops(E)$. Then $S,T\in\Orth(E)$, and $S_\alpha T_\beta\uoconv ST$ in $\obops(E)$. Seven similar statements hold that are obtained by, for each of the three occurrences of unbounded order convergence, either keeping it or replacing it with strong unbounded order convergence.
\end{proposition}

\begin{proof}
	We start with the statement for three occurrences of unbounded order convergence. For this, we first suppose that $S=T=0$.
	
	For $\alpha\in\is$, let $\mathcalalt{P}_\alpha$ be the order projection in $\Orth(E)$ onto the band $B_\alpha$ in $\Orth(E)$ that is generated by $(\abs{S_\alpha}-\Id)^+$. Then $0\leq\mathcalalt{P}_\alpha I\leq \mathcalalt{P}_\alpha \abs{S_\alpha} \leq \abs{S_\alpha}$ by \cite[Lemma~6.9]{deng_de_jeu:2021}. Hence $\mathcalalt{P}_\alpha I\uoconv 0$ in $\obops(E)$, so that also $\mathcalalt{P}_\alpha I\uoconv 0$ in the regular vector sublattice $\Orth(E)$ of $\obops(E)$ by \cite[Theorem~3.2]{gao_troitsky_xanthos:2017}. Since the net $\netgen{\mathcalalt{P_\alpha}I}{\alpha\in\is}$ is order bounded in $\Orth(E)$, we see that
	\begin{equation}\label{eq:uoc1}
	\mathcalalt{P}_\alpha I\oconv 0
	\end{equation}
in $\Orth(E)$. 	Furthermore, since $(\mathcalalt{P}_\alpha \abs{S_\alpha})T_\beta \in B_\alpha$ for $\alpha\in\is,\beta\in\istwo$,  (see \cite[Theorem~2.62]{aliprantis_burkinshaw_POSITIVE_OPERATORS_SPRINGER_REPRINT:2006} or \cref{res:orthisorth}), we also have that $[(\mathcalalt{P}_\alpha \abs{S_\alpha})T_\beta]\wedge I \in B_\alpha$ for $\alpha\in\is,\beta\in\istwo$. Hence
\begin{equation}\label{eq:uoc2}
	[(\mathcalalt{P}_\alpha\abs{S_\alpha})T_\beta]\wedge \Id=\mathcalalt{P}_\alpha \big([(\mathcalalt{P}_\alpha\abs{S_\alpha})T_\beta]\wedge \Id \big)\leq \mathcalalt{P}_\alpha \Id
\end{equation}
for  $\alpha\in\is,\beta\in\istwo$.	

Combining the fact that $\abs{S_\alpha}\leq \Id + \mathcalalt{P}_\alpha \abs{S_\alpha}$ by  \cite[Proposition~6.10(2)]{deng_de_jeu:2021} with \cref{eq:uoc2}, we have, for $\alpha\in\is,\beta\in\istwo$,
\begin{align*}
	\abs{S_\alpha T_\beta}\wedge \Id &\leq (\abs{S_\alpha}\abs{T_\beta})\wedge \Id
	\\&\leq [(\Id + \mathcalalt{P}_\alpha \abs{S_\alpha})\abs{T_\beta}]\wedge \Id
	\\&\leq \abs{T_\beta}\wedge \Id+[(\mathcalalt{P}_\alpha \abs{S_\alpha})\abs{T_\beta}]\wedge \Id
	\\&\leq \abs{T_\beta}\wedge \Id+\mathcalalt{P}_\alpha \Id.
	\end{align*}
The fact that $T_\beta\uoconv 0$ in $\obops(E)$ and then also in $\Orth(E)$, together with \cref{eq:uoc1}, now shows that $\abs{S_\alpha T_\beta}\wedge \Id\oconv 0$ in $\Orth(E)$. Since $\Id$ is a weak order unit of $\Orth(E)$, \cite[Corollary~3.5]{gao_troitsky_xanthos:2017} (or the more general \cite[Proposition~7.4]{deng_de_jeu:2022}) implies that $S_\alpha T_\beta\uoconv 0$ in $\Orth(E)$ and then also in $\obops(E)$.

For the case of general $S$ and $T$, we first note that $S,T\in\Orth(E)$ as a consequence of \cref{res:orth_is_closed}. On writing
\[
S_\alpha T_\beta -TS = (S_\alpha-S)(T_\beta-T) + S_\alpha T + S T _\beta -2TS,
\]
the special case considered above, together with \cref{res:RL_unbounded_order}, then implies that $S_\alpha T_\beta\uoconv ST$ in $\obops(E)$, as desired.

The remaining seven statements follow from the case just established on invoking \cref{res:orth_is_closed}, \cite[Theorem~9.7]{deng_de_jeu:2021}, and \cite[Theorem~3.2]{gao_troitsky_xanthos:2017}.
\end{proof}

The following is now clear from \cref{res:simultaneous_unbounded_order_continuity}, \cite[Theorem~9.7]{deng_de_jeu:2021},  and the simultaneous unbounded order continuity of the lattice operations.

\begin{corollary}\label{res:uoad_is_alg}
	Let $E$ be a Dedekind complete vector lattice, and let $\opalg$ be a subalgebra of $\Orth(E)$. Then the adherences $\uoadh{\opalg}$ and $\SUOadh{\opalg}$ in $\obops(E)$ are equal, and are subalgebras of $\Orth(E)$. When $\opalg$ is a vector lattice subalgebra of $\Orth(E)$, then so is $\uoadh{\opalg}=\SUOadh{\opalg}$.
\end{corollary}

We now show that, neither for $\uoadh{\opalg}$ to be a subalgebra of $\obops(E)$, nor for $\SUOadh{\opalg}$ to be a subalgebra of $\obops(E)$, the condition in \cref{res:uoad_is_alg} that $\opalg\subseteq\Orth(E)$ can be relaxed to $\opalg\subseteq\ocops(E)$.

\begin{example}\label{exam:uo_not_algebra}
Let $E=\ell_1$, and let $\seq{e_n}{n}$ be the standard sequence of unit vectors in $E$. For $i,j\geq 1$, we define $S_{i,j}\in \ocops(E)=\obops(E)$ by setting
\[
S_{i,j}e_n\coloneqq
\begin{cases}e_j&\text{ if } n=i;\\
	0&\text{ if } n\neq i
	\end{cases}
\]
for $n\geq 1$, and we define $T\in \ocops (E)$ by setting
\[
Tx\coloneqq \left(\sum_{i=2}^{\infty}x_i\right) e_3
\]
for $x=\bigvee_{i=1}^{\infty}x_i e_i \in E$. Set $S_n\coloneqq S_{1,2}-S_{1,n+3}$ for $n\geq 1$. It is not hard to check that $T^2=T$, that $S_nT=TS_n=0$ for $n\geq 1$, and that $S_mS_n=0$ for $m,n\geq 1$. Hence $\opalg\coloneqq\Span \{T,\,S_n:n\geq 1\}$ is a subalgebra of $\ocops(E)$.

Using \cite[Theorem~1.51]{aliprantis_burkinshaw_POSITIVE_OPERATORS_SPRINGER_REPRINT:2006}, it is easy to see that $\seq{S_{1,n+3}}{n}$ is a disjoint sequence in $\obops(E)$, so that $S_{1,n+3}\uoconv 0$ in $\obops(E)$ by \cite[Corollary 3.6]{gao_troitsky_xanthos:2017}. Hence $S_n\uoconv S_{1,2}$ in $\obops(E)$, showing that $S_{1,2}\in\uoadh{\opalg}$. Obviously, $T\in\uoadh{\opalg}$. We claim that, however,  $TS_{1,2}$ is not even an element of $\overline{\opalg}^{\uoLt_{\obops(E)}}\supseteq\uoadh{\opalg}$. In order to see this, we observe that $TS_{1,2}=S_{1,3}$ and, using \cite[Theorem~1.51]{aliprantis_burkinshaw_POSITIVE_OPERATORS_SPRINGER_REPRINT:2006}, that $S_{1,3}\perp T$ and $S_{1,3}\perp S_n$ for $n\geq 1$. Hence $TS_{1,2}\perp\opalg$, which implies that $TS_{1,2}\notin\overline{\opalg}^{\uoLt_{\Orth(E)}}$.

For $x=\bigvee_{i=1}^\infty x_i e_i\in\ell_1$, we have $S_{1,n+3}x=x_1e_{n+3}$ for $n\geq 1$. This implies that $S_{1,n+3}\SUOconv 0$ in $\obops(E)$, showing that $S_n\SUOconv S_{1,2}$ in $\obops(E)$. Hence $S_{1,2}\in\SUOadh{\opalg}$. Obviously, $T\in\SUOadh{\opalg}$. We claim that, however,  $TS_{1,2}$ is not even an element of $\adh{\SuoLt_E}{\opalg}\supseteq\SUOadh{\opalg}$. In order to see this, it is sufficient to observe that $TS_{1,2}e_1=e_3\neq 0$ and that $TS_{1,2}e_1\perp Re_1$ for all $R\in\opalg$. This implies that there cannot exist a net $\netgen{R_\alpha}{\alpha\in\is}\subseteq\opalg$ such that $R_\alpha e_1\conv{\uoLt_E} TS_{1,2}e_1$ in $E$, let alone such that $R_\alpha\conv{\SuoLt_E} TS_{1,2}$ in $\obops(E)$.
\end{example}

We turn to closures in a Hausdorff \uoLtop\ and strong closures with respect to a Hausdorff \uoLtop.  We recall once more from \cref{res:all_three_or_none} that either all of $E$, $\Orth(E)$, and $\obops(E)$ admit a Hausdorff \uoLtop, or none does. If they do, then, by general principles (see \cite[Proposition~5.12]{taylor:2019}), $\uoLt_{\Orth(E)}$ is the restriction of $\uoLt_{\obops(E)}$ to $\Orth(E)$.

\begin{proposition}\label{res:simultaneous_uo-Lebesgue_continuity}
Let $E$ be a Dedekind complete vector lattice that admits a \necun\ Hausdorff \uoLtop\ $\uoLt_E$. Suppose that $\opnet\subseteq\Orth(E)$ is a net such that $S_\alpha\conv{\uoLt_{\obops(E)}} S$ in $\obops(E)$ for some $S\in\obops(E)$ and that $\opnettwo\subseteq\Orth(E)$ is a net such that $T_\beta\conv{\uoLt_{\obops(E)}} T$ in $\obops(E)$ for some $T\in\obops(E)$. Then $S,T\in\Orth(E)$, and $S_\alpha T_\beta\conv{\uoLt_{\obops(E)}} ST$ in $\obops(E)$.  Seven similar statements hold that are obtained by, for each of the three occurrences of $\uoLt_{\obops(E)}$ convergence, either keeping it or replacing it with strong $\uoLt_{E}$ convergence.
\end{proposition}

\begin{proof}
	We start with the statement for three occurrences of $\uoLt_{\obops(E)}$ convergence. For this, we first suppose that $S=T=0$.
	
	We can use parts of the proof of \cref{res:simultaneous_unbounded_order_continuity} here. In that proof, it was established that, for $\alpha\in\is$, there exists a band projection $\mathcalalt{P}_\alpha$ in $\Orth(E)$ such that
	\begin{equation}\label{eq:simultaneous_uo-Lebesgue_continuity_1}
	0\leq \mathcalalt{P}_\alpha I\leq \abs{S_\alpha}
	\end{equation}
	and such that
	\begin{equation}\label{eq:simultaneous_uo-Lebesgue_continuity_2}
	\abs{S_\alpha T_\beta}\wedge I\leq \abs{T_\beta}\wedge I +\mathcalalt{P}_\alpha I
	\end{equation}
	for $\beta\in\istwo$.  It follows from \cref{eq:simultaneous_uo-Lebesgue_continuity_1} that also
	\begin{equation*}
			\mathcalalt{P}_\alpha I\conv{\uoLt_{\obops(E)}}0
		\end{equation*}
in $\obops(E)$, and then \cref{eq:simultaneous_uo-Lebesgue_continuity_2} shows that $\abs{S_\alpha T_\beta}\wedge I\conv{\uoLt_{\obops(E)}}0$ in $\obops(E)$, so that also
	\begin{equation}\label{eq:simultaneous_uo-Lebesgue_continuity_3}
\abs{S_\alpha T_\beta}\wedge I\conv{\uoLt_{\Orth(E)}}0
	\end{equation}
 in $\Orth(E)$. It follows from \cite[Theorems~1.38 and~2.45 ]{aliprantis_burkinshaw_POSITIVE_OPERATORS_SPRINGER_REPRINT:2006} that, in the vector lattice $\Orth(E)$, the order adherence of the ideal that is generated by $I$ is the entire space $\Orth(E)$. Hence this ideal is certainly $\uoLt_{\Orth(E)}$- dense in $\Orth(E)$. Since  $\uoLt_{\Orth(E)}$ is an unbounded topology on $\Orth(E)$, it now follows from \cref{eq:simultaneous_uo-Lebesgue_continuity_3} and \cite[Corollary~3.5]{kandic_taylor:2018} that  $S_\alpha T_\beta\conv{\uoLt_{\Orth(E)}}0$ in $\Orth(E)$, and then also $S_\alpha T_\beta\conv{\uoLt_{\obops(E)}}0$ in $\obops(E)$.

	For the case of general $S$ and $T$, we first note that $S,T\in\Orth(E)$ as a consequence of \cref{res:orth_is_closed}. On writing
	\[
	S_\alpha T_\beta -TS = (S_\alpha-S)(T_\beta-T) + S_\alpha T + S T _\beta -2TS,
	\]
	the special case considered above, together with \cref{res:RL_uoLt}, then implies that $S_\alpha T_\beta\conv{\uoLt_{\obops(E)}} ST$ in $\obops(E)$, as desired.
	
	The remaining seven statements follow from the case just established on invoking \cref{res:orth_is_closed} and \cite[Theorem~9.10]{deng_de_jeu:2021}.
\end{proof}

The following is now clear from \cref{res:simultaneous_uo-Lebesgue_continuity} and the simultaneous continuity of the lattice operations with respect to the $\uoLt_{\obops(E)}$ topology.

\begin{corollary}\label{res:uoLt_is_alg}
	Let $E$ be a Dedekind complete vector lattice that admits a \necun\ Hausdorff \uoLtop\  $\uoLt_E$. Suppose that $\opalg$ is a subalgebra of $\Orth(E)$. Then the closure $\overline{\opalg}^{\uoLt_{\obops(E)}}$ in $\obops(E)$ and the adherence $\adh{\SuoLt_E}{\opalg}$ in $\obops(E)$ are equal, and are subalgebras of $\Orth(E)$. When $\opalg$ is a vector lattice subalgebra of $\Orth(E)$, then so is $\overline{\opalg}^{\uoLt_{\obops(E)}}=\adh{\SuoLt_E}{\opalg}$.
\end{corollary}

We now show that, neither for $\overline{\opalg}^{\uoLt_{\obops(E)}}$ to be a subalgebra of $\obops(E)$, nor for $\adh{\SuoLt_E}{\opalg}$ to be a subalgebra of $\obops(E)$, the condition in \cref{res:uoLt_is_alg} that $\opalg\subseteq\Orth(E)$ can be relaxed to $\opalg\subseteq\ocops(E)$.

\begin{example}\label{exam:uoLt_not_algebra}
	We return to the context and notation of \cref{exam:uo_not_algebra}. In that example, we saw that $S_n\uoconv S_{1,2}$ in $\obops(E)$. Then certainly $S_n\conv{\uoLt_{\obops}} S_{1,2}$ in $\obops(E)$, so that both $T$ and $S_{1,2}$ are elements of $\overline{\opalg}^{\uoLt_{\obops(E)}}$. We saw in  \cref{exam:uo_not_algebra}, however, that $TS_{1,2}\notin\overline{\opalg}^{\uoLt_{\obops(E)}}$.
	
	It was also observed that $S_n\SUOconv S_{1,2}$ in $\obops(E)$. Since $E$ is atomic, the unbounded order convergence of a net in $E$ and its convergence in the Hausdorff \uoLtop\ on $E$ are known to coincide (see \cite[Lemma~3.1]{dabboorasad_emelyanov_marabeh:2020} and \cite[Lemma~7.4]{taylor:2019}). Thus also $S_n\conv{\SuoLt_{E}}S_{1,2}$, so that both $T$ and $S_{1,2}$ are elements of $\adh{\SuoLt_E}{\opalg}$. We saw in  \cref{exam:uo_not_algebra}, however, that $TS_{1,2}\notin\adh{\SuoLt_E}{\opalg}$.
	
\end{example}

\section{Equality of adherences of vector sublattices}\label{sec:equality_of_adherences}

\noindent
In this section, we establish the equality of various adherences of vector sublattices with respect to convergence structures under consideration in this paper. We pay special attention to vector sublattices of the orthomorphisms on a Dedekind complete vector lattice. Apart from the intrinsic interest of the results, our research in this direction is also motivated by representation theory. We shall now explain this.

Suppose that $E$ is a vector lattice, and that $\opset$ is a non-empty set of order bounded linear operators on $E$. Typical examples to keep in mind are those where $\opset$ is a group of order automorphisms of $E$ (this arises naturally when considering positive representations of groups on vector lattices), or where $\opset$ is a (vector lattice) algebra of order bounded linear operators (this arises naturally when considering positive representations of (vector lattice) algebras on vector lattices). One of the main issues in general representation theory is to investigate the possible decompositions of a module into submodules. In our case, this is asking for decompositions $E=F_1\oplus F_2$ as an order direct sum of vector sublattices $F_1$ and $F_2$ that are both invariant under $\opset$. It is well known (see \cite[Theorem~11.3]{zaanen_INTRODUCTION_TO_OPERATOR_THEORY_IN_RIESZ_SPACES:1997} for an even stronger result) that $F_1$ and $F_2$ are then projection bands that are each other's disjoint complements. Their respective order projections then commute with all elements of $\opset$. Conversely, when an order projection has this property, then $E$ is the order direct sum of its range and its kernel, and both are invariant under $\opset$. All in all, the decomposition question for the action of $\opset$ on $E$ is the same as asking for the order projections on $E$ that commute with $\opset$. This makes it natural to ask for the commutant of $\opset$ in $\Orth(E)$, where these order projections reside. This commutant is obviously an associative subalgebra of $\Orth(E)$. Somewhat surprisingly, it is quite often also a vector sublattice of $\Orth(E)$. For example, this is always true for Banach lattices, in which case the operators in $\opset$ need not even be regular. Being bounded is enough, as is shown by the following result, for which the Banach lattice need not even be Dedekind complete.

\begin{theorem}\label{res:commutant_of_bounded_operators_is_lattice}
	Let $E$ be a Banach lattice, and let $\opset$ be a non-empty set of bounded linear operators on $E$. Then the commutant
	\[
	\orthcom{\opset}\coloneqq \{T\in\Orth(E) : TS=ST\text{ for all }S\in\opset\}
	\]
	of $\opset$ in $\Orth(E)$ is a Banach \fsubalg\ of $\Orth(E)$ that contains the identity operator $\Id$ as a strong order unit; here $\Orth(E)$ is supplied with the coinciding operator norm and order unit norm $\norm{\,\cdot\,}_I$.
\end{theorem}

\begin{proof}
It is obvious that $\orthcom{\opset}$ is an associative subalgebra of $\Orth(E)$ that contains $\Id$ and that is closed with respect to the coinciding operator norm and order unit norm $\norm{\,\cdot\,}_I$. An appeal to \cite[Theorem~6.1]{deng_de_jeu:2021} then finishes the proof.
\end{proof}

For Dedekind complete vector lattices, we have the following.

\begin{theorem}\label{res:commutant_of_order_continuous_operators_is_lattice}
	Let $E$ be a Dedekind complete vector lattice, let $\opset$ be a non-empty subset of $\obops(E)$, and let $B_{\opset}$ be the band in $\obops(E)$ that is generated by $\opset$. Then the commutant
	\begin{align*}
	&\orthcom{\opset}\coloneqq \{T\in\Orth(E) : TS=ST\text{ for all }S\in\opset\}
	\intertext{of $\opset$ in $\Orth(E)$ and the commutant}
	&\orthcom{(B_\opset)}\coloneqq \{T\in\Orth(E) : TS=ST\text{ for all }S\in B_\opset\}
	\end{align*}
	of $B_\opset$ in $\Orth(E)$ are equal. This common commutant in $\Orth(E)$ is a vector lattice subalgebra of $\Orth(E)$ that contains the identity operator $\Id$ as a weak order unit.
	
	Suppose that $\opset\subseteq\ocops(E)$. Then:
	\begin{enumerate}
		\item\label{res:commutant_of_order_continuous_operators_is_lattice_1} $\orthcom{\opset}$ is an order closed vector sublattice of every regular vector sublattice of $\obops(E)$ containing $\orthcom{\opset}$;
		\item\label{res:commutant_of_order_continuous_operators_is_lattice_2} $\orthcom{\opset}$ is a regular vector sublattice of every Dedekind complete regular vector sublattice of $\obops(E)$ containing $\orthcom{\opset}$;
		\item\label{res:commutant_of_order_continuous_operators_is_lattice_3} $\orthcom{\opset}$ is a Dedekind complete vector lattice.
	\end{enumerate}
\end{theorem}

\begin{proof}
	
For $T\in\Orth(E)$, we let $\Le{T}$ (resp.\ $\Ri{T}$) denote the corresponding left (resp.\ right) multiplication operator on $\obops(E)$. Since \cref{res:orthisorth} implies that $\Le{T}-\Ri{T}$ is an orthomorphism on $\obops(E)$, its kernel is a band in $\obops(E)$ by \cite[Theorem~2.48]{aliprantis_burkinshaw_POSITIVE_OPERATORS_SPRINGER_REPRINT:2006}. It follows from this that the commutants of $\opset$ and $B_\opset$ in $\Orth(E)$ are equal.

We shall now show that this common commutant in $\Orth(E)$ is a vector sublattice of $\Orth(E)$. In view of what we have already established, we may suppose that $\opset$ consists of one positive operator $S$ on $E$. We shall show that for $T_1, T_2\in\Orth(E)$, $T_1\vee T_2$ commutes with $S$ whenever $T_1$ and $T_2$ do. Obviously, $(T_1\vee T_2)S=\lambda_{T_1\vee T_2}(S)$ which, by part~\ref{res:orthisorth2} of \cref{res:orthisorth,} equals $(\lambda_{T_1}\vee\lambda_{T_2})(S)$. Using once more from part~\ref{res:orthisorth2} of \cref{res:orthisorth} that left multiplications by elements of $\Orth(E)$ are orthomorphisms on $\obops(E)$, \cite[Theorem~2.43]{aliprantis_burkinshaw_POSITIVE_OPERATORS_SPRINGER_REPRINT:2006} implies that $(\lambda_{T_1}\vee\lambda_{T_2})(S)=\lambda_{T_1}(S)\vee\lambda_{T_2}(S)=(T_1S)\vee(T_2S)$. As a consequence of the assumption, this equals $(ST_1)\vee (ST_2)$. By a reasoning similar to the one just given, but now using part~\ref{res:orthisorth1} of \cref{res:orthisorth}, this equals $S(T_1\vee T_2)$. Hence $\orthcom{\opset}$ is a vector sublattice of $\Orth(E)$.

It is clear that $\orthcom{\opset}$ is an associative subalgebra of $\Orth(E)$ containing $\Id$ and that $\Id$, which is a weak order unit of $\Orth(E)$, is also one of $\orthcom{\opset}$.

We turn to the remaining statements when $S\subseteq\ocops(E)$. Suppose that $\netgen{T_\alpha}{\alpha\in\is}$ is a net in $\orthcom{\opset}$, that $T\in\obops(E)$, and that $T_\alpha\oconv T$ in $\obops(E)$. Then certainly $T\in\Orth(E)$. Using that $\opset\subseteq\ocops(E)$, it follows from \cref{res:RL_order} that $T$ commutes with all elements of $\opset$. Hence $T\in\orthcom{\opset}$, and we conclude that $\orthcom{\opset}$ is an order closed vector sublattice of $\obops(E)$. Obviously, it is then also order closed in every regular vector sublattice of $\obops(E)$ containing it. We have thus established part~\ref{res:commutant_of_order_continuous_operators_is_lattice_1}.

Take a Dedekind complete regular vector sublattice $\oplattwo$ of $\obops(E)$ that contains $\orthcom{\opset}$. Since we know that $\orthcom{\opset}$ is order closed in $\oplattwo$, \cite[p.~303]{luxemburg_de_pagter:2005a} shows that $\orthcom{\opset}$ is a complete vector sublattice of $\oplattwo$ as this notion is defined on \cite[p.~295-296]{luxemburg_de_pagter:2005a}. It then follows from \cite[p.~296]{luxemburg_de_pagter:2005a} that $\orthcom{\opset}$ is a regular vector sublattice of $\oplattwo$ and, on taking $\oplattwo=\obops(E)$, also that $\orthcom{\opset}$ is Dedekind complete.
\end{proof}

\begin{remark}
As a special case of \cref{res:commutant_of_order_continuous_operators_is_lattice},  it was already established in \cite[Lemma~8.9]{de_rijk_THESIS:2012} that, for a Dedekind complete vector lattice $E$, $\orthcom{\opset}$ is an order closed vector lattice subalgebra of $\Orth(E)$ for every vector sublattice $\opset$ of $\ocops(E)$.
\end{remark}

In \cref{res:commutant_of_order_continuous_operators_is_lattice}, when $\opset\subseteq\ocops(E)$, then the vector lattice $\orthcom{\opset}$ is a Dedekind complete vector lattice with the identity operator $\Id$ as a weak order unit. The unbounded version of Freudenthal's spectral theorem (see \cite[Theorem~40.3]{luxemburg_zaanen_RIESZ_SPACES_VOLUME_I:1971}, for example) then shows that an arbitrary element $T\in\orthcom{\opset}$ is an order limit of a sequence of linear combinations of the components of $\Id$ in $\orthcom{\opset}$. Since the latter are precisely the order projections that commute with $\opset$ we see that, in this case, $\orthcom{\opset}$ does not only contain all information about the collection of bands in $E$ that reduce $\opset$, but that it is also completely determined by this collection.

On a later occasion, we shall report more elaborately on the procedures of taking commutants and also of taking bicommutants in the context of operators on vector lattices and Banach lattices, as well as on their relations with reducing projection bands for sets of operators; see \cite{deng_de_jeu_UNPUBLISHED:bicommutant_2,deng_de_jeu_UNPUBLISHED:bicommutant_1}. For the moment, we content ourselves with the general observation that the study of vector lattice subalgebras of the orthomorphisms is relevant for representation theory on vector lattices.

We shall now set out to study one particular aspect of this, namely, the equality of the adherences of vector sublattices of the orthomorphisms with respect to several of the convergence structures under consideration in this paper. Although from a representation theoretical point of view it would be natural to require that they also be associative subalgebras, this does, so far, not appear to be relevant for these issues. Such results on equal adherences can then also be obtained for associative subalgebras of the orthomorphisms on a Banach lattice, as a consequence of the fact that their norm closures in the orthomorphisms are. in fact, vector sublattices to which the previous results can be applied.

Regarding the results below that are given for vector sublattices of the orthomorphisms, we recall that, for a Dedekind complete vector lattice, several adherences coincide for subsets of the orthomorphisms. Indeed, since, for nets of orthomorphisms, unbounded order convergence coincides with strong unbounded order convergence, and since the convergence in a possible Hausdorff \uoLtop\ coincides with the corresponding strong convergence, the corresponding adherences of subsets of the orthomorphisms are also equal. The same holds for sequential adherences. \emph{For reasons of brevity, we have refrained from including these `obviously also equal' adherences in the statements.}

Although our motivation leads us to study vector sublattice of the orthomorphisms, the results as we shall derive them for these are actually consequences of more general statements for vector lattices that need not even consist of operators. These are of interest in their own right. Other such results are \cite[Theorem~8.8]{deng_de_jeu:2022}, \cite[Theorem~2.13]{gao_leung:2018}, and \cite[Proposition 2.12]{taylor:2019}.

We start by establishing results showing that the closures of vector sublattices (or associative subalgebras) in a possible Hausdorff \uoLtop\ coincide with their closures in other linear topologies on the vector lattices (or associative algebras) under consideration. These are based on the following result, which is established in the first paragraph of the proof of \cite[Theorem~8.8]{deng_de_jeu:2022}. For the definition of the absolute weak topology $\abs{\sigma}(E, I)$ on $E$ that occurs in it we refer to \cite[p.~63]{aliprantis_burkinshaw_LOCALLY_SOLID_RIESZ_SPACES_WITH_APPLICATIONS_TO_ECONOMICS_SECOND_EDITION:2003}.

\begin{proposition}\label{res:uoLt_closure_separating_order_dual}
	Let $E$ be a vector lattice such that $\ocdual{E}$ separates the points of $E$. Then $E$ admits a \necun\ Hausdorff \uoLtop\ $\uoLt_E$. Take an ideal $I$ of $\ocdual{E}$ that separates the points of $E$, and take a vector sublattice $F$ of $E$. Then
	\[
	\overline{F}^{\uoLt_{E}}=\overline{F}^{\sigma(E, I)}=\overline{F}^{\abs{\sigma}(E, I)}
	\]
	in $E$.
\end{proposition}

In the following consequence of \cref{res:uoLt_closure_separating_order_dual}, the lattice $\oplattwo$ of operators can be taken to be $\Orth(E)$.

\begin{corollary}\label{res:operators_uoLt_closure_separating_order_dual}
	Let $E$ be a Dedekind complete vector lattice such that $\ocdual{E}$ separates the points of $E$, and let $\oplat $ be a regular vector sublattice of $\obops(E)$. Then $\ocdual{\oplat}$ separates the points of $\oplat$, and $\oplat$ admits a \necun\ Hausdorff \uoLtop\ $\uoLt_{\oplat}$. Take an ideal $I$ of $\ocdual{\oplat}$ that separates the points of $\oplat$, and take a vector sublattice $\oplattwo$ of $\oplat$. Then
	\[
	\overline{\oplattwo}^{\uoLt_{\oplat}}=\overline{\oplattwo}^{\sigma(\oplat, I)}=\overline{\oplattwo}^{\abs{\sigma}(\oplat, I)}
	\]
	in $\oplat$.
\end{corollary}

\begin{proof}
For $\varphi\in\ocdual{E}$ and $x\in E$, define the order bounded linear functional on $\oplat$ by setting $\Phi_{\varphi,x}(T)\coloneqq \varphi(Tx)$ for $T\in\oplat$. Since $\oplat$ is a regular vector sublattice of $\obops(E)$, an appeal to \cite[Lemma~4.1]{deng_de_jeu:2021} shows that $\Phi_{\varphi,x}\in\ocdual{\oplat}$. Is then clear that $\ocdual{\oplat}$ separates the points of $\oplat$. Now \cref{res:uoLt_closure_separating_order_dual} can be applied with $E$ replaced by $\oplat$ and $F$ by $\oplattwo$.
\end{proof}	

\cref{res:uoLt_closure_separating_order_dual} is also used in the proof of the following.

\begin{theorem}\label{res:f-algebra_uoLt_closure_separating_order_dual}
	Let $\opalg$ be a unital \falg\ such that its identity element $e$ is also a strong order unit of $\opalg$, and such that it is complete in the submultiplicative order unit norm $\norm{\,\cdot\,}_e$ on $\opalg$. Suppose that $\ocdual{\opalg}$ separates the points of $\opalg$. Then $\opalg$ admits a \necun\ Hausdorff \uoLtop\ $\uoLt_{\opalg}$. Take an ideal $I$ of $\ocdual{\opalg}$ that separates the points of $\opalg$, and take a \uppars{not necessarily unital} associative subalgebra $\opalgtwo$ of $E$. Then
	\begin{align}\label{eq:f-algebra_uoLT_closure_separating_order_dual}
		\begin{split}
			&\overline{\opalgtwo}^{\uoLt_{\opalg}}=\overline{\opalgtwo}^{\sigma(\opalg, I)}=\overline{\opalgtwo}^{\abs{\sigma}(\opalg, I)}=\\
			&\overline{\overline{\opalgtwo}^{\norm{\,\cdot\,}_e}}^{\uoLt_{\opalg}}
			=\overline{\overline{\opalgtwo}^{\norm{\,\cdot\,}_e}}^{\sigma(\opalg, I)}=
			\overline{\overline{\opalgtwo}^{\norm{\,\cdot\,}_e}}^{\abs{\sigma}(\opalg, I)}		
		\end{split}
	\end{align}
	in $\opalg$.
\end{theorem}

Before giving the proof, we mention the following fact that is easily verified. Suppose that $X$ is a topological space that is supplied with two topologies $\tau_1$ and $\tau_2$, where $\tau_2$ is weaker than $\tau_1$. Then  $\overline{\overline{S}^{\tau_1}}^{\tau_2}=\overline{S}^{\tau_2}$ for every subset $S$ of $X$.

\begin{proof}
It follows from \cite[Theorem~6.1]{deng_de_jeu:2021} that $\overline{\opalgtwo}^{\norm{\,\cdot\,}_e}$ is a Banach \fsubalg\ of $\opalg$. Being a vector sublattice of $\opalg$, \cref{res:uoLt_closure_separating_order_dual} shows that the sets in the second line of \cref{eq:f-algebra_uoLT_closure_separating_order_dual} are equal. Since the convergence of a net in the order unit norm $\norm{\,\cdot\,}_e$  implies its order convergence to the same limit (and then also its convergence in $\uoLt_\opalg$ to the same limit), we are done by an appeal to the remark preceding the proof.
\end{proof}	

The following is now clear from \cref{res:f-algebra_uoLt_closure_separating_order_dual} and the argument in the proof of \cref{res:operators_uoLt_closure_separating_order_dual}.

\begin{corollary}\label{res:algebra_in_orth_uoLt_closure_separating_order_dual}
	Let $E$ be a Dedekind complete Banach lattice. Suppose that $\ocdual{E}$ separates the points of $E$. Then $\ocdual{\Orth(E)}$ separates the points of $\Orth(E)$, and $\Orth(E)$ admits a \necun\ Hausdorff \uoLtop\ $\uoLt_{\Orth(E)}$. Take an ideal $I$ of $\ocdual{\Orth(E)}$ that separates the points of $\Orth(E)$, and take a \uppars{not necessarily unital} associative subalgebra $\opalg$ of $\Orth(E)$. Then
	\begin{align*}
	&\overline{\opalg}^{\uoLt_{\Orth(E)}}=\overline{\opalg}^{\sigma(\Orth(E), I)}=\overline{\opalg}^{\abs{\sigma}(\Orth(E), I)}=\\
	&\overline{\overline{\opalg}^{\norm{\,\cdot\,}}}^{\uoLt_{\Orth(E)}}=\overline{\overline{\opalg}^{\norm{\,\cdot\,}}}^{\sigma(\Orth(E), I)}=\overline{\overline{\opalg}^{\norm{\,\cdot\,}}}^{\abs{\sigma}(\Orth(E), I)}
	\end{align*}
	in $\Orth(E)$;  here $\norm{\,\cdot\,}$ denotes the coinciding operator norm, order unit norm with respect to the identity operator, and regular norm on $\Orth(E)$.
\end{corollary}

We shall now continue by establishing results showing that the closures of vector sublattices (or associative subalgebras) in a possible Hausdorff \uoLtop\ coincide with their adherences with respect to various convergence structures on the enveloping vector lattices (or vector lattice algebras) under consideration in this paper.

 Needless to say, under appropriate conditions, `topological' results as obtained above may apply at the same time as `adherence' results to be obtained below. \emph{For reasons of brevity, we have refrained from formulating such `combined' results.}

Let us also notice at this point that the results below imply that the adherences of vector sublattices that occur in the statements are closed with respect to the pertinent convergence structures. Indeed, these adherences are set maps that map vector sublattices to vector sublattices. When they agree on vector sublattices with the topological closure operator that is supplied by the Hausdorff \uoLtop, then they, too, are idempotent. For example, the unbounded order adherence of the vector sublattice $F$ in \cref{res:uoLt_closure_weak_unit_countable_sup_property}, is unbounded order closed. \emph{For reasons of brevity, we have refrained from including such consequences in the results.}

We start by considering two cases where the enveloping vector lattices have weak order units.

\begin{proposition}\label{res:uoLt_closure_weak_unit_countable_sup_property}
	Let $E$ be a Dedekind complete vector lattice with the countable sup property and a weak order unit. Suppose that $E$ admits a \necun\ Hausdorff \uoLtop \ $\uoLt_E$. Let $F$ be a vector sublattice of $E$. Then
	\[
	\overline{F}^{\uoLt_E}=\suoadh{F}=\uoadh{F}
	\]
	in $E$.
\end{proposition}

\begin{proof}
Clearly, we have $\suoadh{F}\subseteq \uoadh{F}\subseteq \overline{F}^{\uoLt_E}$. Let $e$ be a positive weak order unit of $E$. Take $x\in \overline{F}^{\uoLt_E}$. There exists a net $\net$ in $F$ with $x_\alpha\conv{\uoLt_E} x$. Then $\abs{x_\alpha-x}\wedge e\conv{\uoLt_E} 0$, and we conclude from \cite[Theorem~4.19]{aliprantis_burkinshaw_LOCALLY_SOLID_RIESZ_SPACES_WITH_APPLICATIONS_TO_ECONOMICS_SECOND_EDITION:2003} that there exists an increasing sequence $\seq{\alpha_n}{n}$ of indices in $\is$ such that $\abs{x_{\alpha_n}-x}\wedge e\oconv 0$ in $E$. An appeal to \cite[Lemma~3.2]{gao_xanthos:2014} shows that $x_{\alpha_n}\uoconv x$ in $E$. Hence $x\in \suoadh{F}$. We conclude that $\overline{F}^{\uoLt_E}\subseteq \suoadh{F}$.
\end{proof}

On combining  \cref{res:all_three_or_none}, \cref{res:uoLt_closure_weak_unit_countable_sup_property}, and \cite[Proposition~6.5]{deng_de_jeu:2021}, the following is easily obtained. We recall that a subset of a vector lattice is said to be an order basis when the band that it generates is the whole vector lattice.

\begin{corollary}\label{res:operators_uoLt_closure_weak_unit_countable_sup_property}
	Let $E$ be a Dedekind complete vector lattice with the countable sup property and an at most countably infinite order basis. Suppose that $E$ admits a \necun\ Hausdorff \uoLtop. Then $\Orth(E)$ admits a \necun\ Hausdorff \uoLtop\ $\uoLt_{\Orth(E)}$. Let $\oplat$ be a vector sublattice of $\Orth(E)$. Then
	\[
	\overline{\oplat}^{\uoLt_{\Orth(E)}}=\suoadh{\oplat}=\uoadh{\oplat}
	\]
	in $\Orth(E)$.
\end{corollary}

We continue by considering cases where the enveloping vector lattice (or vector lattice algebra) has a strong order unit.

It is known that the \oadhtext\ of a vector sublattice of a De\-de\-kind complete Banach lattice $E$ with a strong order unit can be a proper sublattice of its \uoadhtext; see \cite[Lemma~2.6]{gao_leung:2018} for details. When the vector sublattice contains a strong order unit of $E$, however, then this cannot occur, not even in general vector lattices. This is shown by the following preparatory result.

\begin{lemma}\label{res:o=uo}
	Let $E$ be a vector lattice with a strong order unit. Suppose that $F$ is a vector sublattice of $E$ that contains a strong order unit of $E$. Then $\oadh{F}=\uoadh{F}$ and $\soadh{F}=\suoadh{F}$ in $E$.
\end{lemma}

\begin{proof}
	 We prove that $\oadh{F}=\uoadh{F}$. It is clear that $\oadh{F}\subseteq\uoadh{F}$. For the reverse inclusion, we choose a positive strong order unit $e$ of $E$ such that $e\in F$. Take $x\in \uoadh{F}$, and let $\net$ be a net in $F$ such that $x_\alpha\uoconv x$ in $E$. There exists a $\lambda\in\RR_{\geq 0}$ such that $\abs{x}\leq \lambda e$. For $\alpha\in\is$, set $y_\alpha\coloneqq (-\lambda e\vee x_\alpha) \wedge \lambda e$. Clearly, $(y_\alpha)_\alpha\subseteq F$ and $y_\alpha\uoconv (-\lambda e\vee x)\wedge \lambda e=x$. Since the net $\netgen{y_\alpha}{\alpha\in\is}$ is order bounded in $E$, we have that $y_\alpha\oconv x$ in $E$. Hence $x\in\oadh{F}$. We conclude that $\uoadh{F}\subseteq\oadh{F}$.
	
	 The proof for the sequential adherences is a special case of the above one.
\end{proof}

\begin{remark}
	For comparison, we recall that, for a \emph{regular} vector sublattice $F$ of a  vector lattice $E$, it is always the case that $\oadh{F} =\uoadh{F}$ in $E$, and that these coinciding subsets are order closed subsets of $E$; see \cite[Theorem~2.13]{gao_leung:2018}. For this to hold, no assumptions on $E$ are necessary.
\end{remark}

The following is immediate from \cref{res:uoLt_closure_weak_unit_countable_sup_property} and \cref{res:o=uo}.

\begin{theorem}\label{res:uoLt_closure_strong_unit_countable_sup_property}
	Let $E$ be a Dedekind complete vector lattice with the countable sup property and a strong order unit. Suppose that $E$ admits a \necun\ Hausdorff \uoLtop \ $\uoLt_E$. Let $F$ be a vector sublattice of $E$ that contains a strong order unit of $E$. Then	
	\[
	\overline{F}^{\uoLt_E}=\soadh{F}=\oadh{F}=\suoadh{F}=\uoadh{F}
	\]
	in $E$.
\end{theorem}

The following result follows from the combination of \cref{res:all_three_or_none}, \cref{res:uoLt_closure_strong_unit_countable_sup_property}, and \cite[Proposition~6.5]{deng_de_jeu:2021}. In view of \cite[Proposition~6.5]{deng_de_jeu:2021}, the natural condition to include is that $E$ have an at most countably infinite order basis, but it is easily verified fact that, for a Banach lattice, this property is equivalent to having a weak order unit.

\begin{corollary}\label{res:operators_uoLt_closure_strong_unit_countable_sup_property}
	Let $E$ be a Dedekind complete Banach lattice with the countable sup property and a weak order unit. Suppose that $E$ admits a \necun\ Hausdorff \uoLtop. Then $\Orth(E)$ admits a \necun\ Hausdorff \uoLtop\ $\uoLt_{\Orth(E)}$. Let $\oplat$ be a vector sublattice of $\Orth(E)$ that contains a strong order unit of $\Orth(E)$. Then	
	\[
	\overline{\oplat}^{\uoLt_{\Orth(E)}}=\soadh{\oplat}=\oadh{\oplat}=\suoadh{\oplat}=\uoadh{\oplat}
	\]
	in $\Orth(E)$.
\end{corollary}

We now turn to closures and adherences of associative subalgebras of a class of \falgs\ with strong order units. For this, we need the following preparatory result.

\begin{lemma}\label{res:normo=o}
	Let $E$ be a Banach lattice, and let $A$ be a subset of $E$. Then $\soadh{A}=\soadh{\overline{A}}$ in $E$,  where $\overline{A}$ denotes the norm closure of $A$.
\end{lemma}

\begin{proof}
	We need to prove only that $\soadh{\overline{A}}\subseteq \soadh{A}$. For this, we may suppose that $A\neq\emptyset$.
	Take $x\in \soadh{\overline{A}}$ and a sequence $\seq{x_n}{n}$ in $\overline{A}$ such that $x_n \conv{\sigma\o} x$ in $E$. For $n\geq 1$, take an $y_n\in A$ such that $\norm{y_n-x_n}\leq 2^{-n}$. For $n\geq 1$, define $z_n$ by setting $z_n\coloneqq \sum_{m=n}^{\infty} \abs{y_m-x_m}$, which is meaningful since the series is absolutely convergent. It is clear that $z_n\downarrow$. Since $\norm{z_n}\leq 2^{-n+1}$, we have $z_n\downarrow 0$ in $E$. The fact that $\abs{y_n-x_n}\leq z_n$ for $n\geq 1$ then shows that $\abs{y_n-x_n}\conv{\sigma\o} 0$ in $E$.
	From
	\[
	0\leq \abs{y_n-x}\leq \abs{y_n-x_n}+\abs{x_n-x}\conv{\sigma\o} 0,
	\]
	we then see that $y_n\conv{\sigma\o} x$ in $E$. Hence $x\in\soadh{A}$, as desired.
\end{proof}

\begin{theorem}\label{res:uoLt_closure_algebra_strong_unit_countable_sup_property}
	Let $\opalg$ be a Dedekind complete unital \falg\ with the countable sup property, such that its identity element $e$ is also a strong order unit of $\opalg$, and such that it is complete in the submultiplicative order unit norm $\norm{\,\cdot\,}_e$ on $\opalg$. Suppose that $\opalg$ admits a \necun\ Hausdorff \uoLtop \ $\uoLt_{\opalg}$. Let $\opalgtwo$ be an associative subalgebra of $\opalg$ such that $\overline{\opalgtwo}^{{\norm{\,\cdot\,}_e}}$  contains a strong order unit of $\opalg$. Then
\begin{align}
\begin{split}\label{eq:ten_sets_equal}	
	&\overline{\opalgtwo}^{\uoLt_{\opalg}}=\soadh{\opalgtwo}=\oadh{\opalgtwo}=\suoadh{\opalgtwo}=\uoadh{\opalgtwo}=	\\
	 &\overline{\overline{\opalgtwo}^{{\norm{\,\cdot\,}_e}}}^{\uoLt_{\opalg}}=\soadh{\overline{\opalgtwo}^{{\norm{\,\cdot\,}_e}}}=\oadh{\overline{\opalgtwo}^{{\norm{\,\cdot\,}_e}}}=\suoadh{\overline{\opalgtwo}^{{\norm{\,\cdot\,}_e}}}=\uoadh{\overline{\opalgtwo}^{{\norm{\,\cdot\,}_e}}}	 
\end{split}
\end{align}
	in $\opalg$.
\end{theorem}

\begin{proof}
	We know from \cite[Theorem~6.1]{deng_de_jeu:2021} that $\overline{\opalgtwo}^{{\norm{\,\cdot\,}_e}}$ is a Banach \fsubalg\ of $\opalg$. Then \cref{res:uoLt_closure_strong_unit_countable_sup_property} shows that all equalities in the second line of \cref{eq:ten_sets_equal} hold. Furthermore, it is obvious that
	\[
	\soadh{\opalgtwo}\subseteq\oadh{\opalgtwo}\subseteq\uoadh{\opalgtwo}\subseteq\overline{\opalgtwo}^{\uoLt_{\opalg}}
	\]
	and that
	\[
	\soadh{\opalgtwo}\subseteq\suoadh{\opalgtwo}\subseteq\overline{\opalgtwo}^{\uoLt_{\opalg}}.
	\]
	Using that $\soadh{\opalgtwo}=\soadh{\overline{\opalgtwo}^{{\norm{\,\cdot\,}_e}}}$ by \cref{res:normo=o} and that\textemdash see the proof of \cref{res:f-algebra_uoLt_closure_separating_order_dual}\textemdash we also know that $\overline{\opalgtwo}^{\uoLt_{\opalg}}=\overline{\overline{\opalgtwo}^{{\norm{\,\cdot\,}_e}}}^{\uoLt_{\opalg}}$, it then follows that all sets in \cref{eq:ten_sets_equal} are equal.
	\end{proof}

The following is now clear from \cref{res:all_three_or_none}, \cref{res:uoLt_closure_algebra_strong_unit_countable_sup_property}, and \cite[Proposition~6.5]{deng_de_jeu:2021}.

\begin{corollary}\label{res:uoLt_closure_algebra_othormorphisms_strong_unit_countable_sup_property}
	Let $E$ be a Dedekind complete Banach lattice with the countable sup property and a weak order unit. Suppose that $E$ admits a \necun\ Hausdorff \uoLtop. Then $\Orth(E)$ admits a \necun\ Hausdorff \uoLtop\ $\uoLt_{\Orth(E)}$. Let $\opalg$ be an associative subalgebra of $\Orth(E)$ such that $\overline{\opalg}^{{\norm{\,\cdot\,}}}$ contains a strong order unit of $\Orth(E)$. Then
	\begin{align}
		\begin{split}
			&\overline{\opalg}^{\uoLt_{\opalg}}=\soadh{\opalg}=\oadh{\opalg}=\suoadh{\opalg}=\uoadh{\opalg}=	\\
			 &\overline{\overline{\opalg}^{{\norm{\,\cdot\,}}}}^{\uoLt_{\opalg}}=\soadh{\overline{\opalg}^{{\norm{\,\cdot\,}}}}=\oadh{\overline{\opalg}^{{\norm{\,\cdot\,}}}}=\suoadh{\overline{\opalg}^{{\norm{\,\cdot\,}}}}=\uoadh{\overline{\opalg}^{{\norm{\,\cdot\,}}}}	 
		\end{split}
	\end{align}
	in $\Orth(E)$; here $\norm{\,\cdot\,}$ denotes the coinciding operator norm, order unit norm with respect to the identity operator, and regular norm on $\Orth(E)$.
\end{corollary}


\subsection*{Acknowledgements} During this research, the first author was supported by a grant of China Scholarship Council (CSC). The authors thank Ben de Pagter for pointing out that a (not necessarily unital) positive algebra homomorphism between two unital Archimedean \falgs\ is automatically a vector lattice homomorphism, which has led to \cref{res:orthisorth_linear},  \cref{res:orthisorth}, and \cref{res:commutant_of_order_continuous_operators_is_lattice} being in a stronger form than in the original manuscript. They also thank the anonymous referee for their careful reading of the manuscript and the many detailed suggestions for improvement of the presentation.


\urlstyle{same}

\bibliographystyle{plain}
\bibliography{general_bibliography}

\def\cprime{$'$}
  \def\lfhook#1{\setbox0=\hbox{#1}{\ooalign{\hidewidth\lower1.5ex\hbox{'}\hidewidth\crcr\unhbox0}}}
\begin{thebibliography}{10}

\bibitem{alekhno:2012}
E.A. Alekhno.
\newblock The irreducibility in ordered {B}anach algebras.
\newblock {\em Positivity}, 16(1):143--176, 2012.

\bibitem{alekhno:2017}
E.A. Alekhno.
\newblock The order continuity in ordered algebras.
\newblock {\em Positivity}, 21(2):539--574, 2017.

\bibitem{aliprantis_burkinshaw_LOCALLY_SOLID_RIESZ_SPACES_WITH_APPLICATIONS_TO_ECONOMICS_SECOND_EDITION:2003}
C.D. Aliprantis and O.~Burkinshaw.
\newblock {\em Locally solid {R}iesz spaces with applications to economics},
  volume 105 of {\em Mathematical Surveys and Monographs}.
\newblock American Mathematical Society, Providence, RI, second edition, 2003.

\bibitem{aliprantis_burkinshaw_POSITIVE_OPERATORS_SPRINGER_REPRINT:2006}
C.D. Aliprantis and O.~Burkinshaw.
\newblock {\em Positive operators}.
\newblock Springer, Dordrecht, 2006.
\newblock Reprint of the 1985 original.

\bibitem{beattie_butzmann_CONVERGENCE_STRUCTURES_AND_APPLICATIONS_TO_FUNCTIONAL_ANALYSIS:2002}
R.~Beattie and H.-P. Butzmann.
\newblock {\em Convergence structures and applications to functional analysis}.
\newblock Kluwer Academic Publishers, Dordrecht, 2002.

\bibitem{bogachev_MEASURE_THEORY_VOLUME_II:2007}
V.I. Bogachev.
\newblock {\em Measure theory. {V}ol. {II}}.
\newblock Springer-Verlag, Berlin, 2007.

\bibitem{chen_schep:2016}
J.X. Chen and A.R. Schep.
\newblock Two-sided multiplication operators on the space of regular operators.
\newblock {\em Proc. Amer. Math. Soc.}, 144:2495--2501, 2016.

\bibitem{conradie:2005}
J.~Conradie.
\newblock The coarsest {H}ausdorff {L}ebesgue topology.
\newblock {\em Quaest. Math.}, 28(3):287--304, 2005.

\bibitem{dabboorasad_emelyanov_marabeh:2020}
Y.~Dabboorasad, E.~Emelyanov, and M.~Marabeh.
\newblock Order convergence is not topological in infinite-dimensional vector
  lattices.
\newblock {\em Uzbek Math. J.}, 64(1):159--166, 2020.

\bibitem{dales_de_jeu:2019}
H.G. Dales and M.~de~Jeu.
\newblock Lattice homomorphisms in harmonic analysis.
\newblock In {\em Positivity and noncommutative analysis}, Trends Math., pages
  79--129. Birkh\"{a}user/Springer, Cham, 2019.

\bibitem{de_pagter_THESIS:1981}
B.~de~Pagter.
\newblock {\em $f\!$-Algebras and orthomorphisms}.
\newblock PhD thesis, Leiden University, Leiden, 1981.

\bibitem{de_rijk_THESIS:2012}
B.~de~Rijk.
\newblock The order bicommutant.
\newblock Master's thesis, Leiden University, Leiden, 2012.
\newblock Available at
  \url{https://www.universiteitleiden.nl/binaries/content/assets/science/mi/scripties/derijkmaster.pdf}.

\bibitem{deng_de_jeu_UNPUBLISHED:bicommutant_2}
Y.~Deng and M.~de~Jeu.
\newblock Bicommutant theorems for algebras of orthomorphisms.
\newblock In preparation.

\bibitem{deng_de_jeu_UNPUBLISHED:bicommutant_1}
Y.~Deng and M.~de~Jeu.
\newblock Topological bicommutant theorems in vector lattices and {B}anach
  lattices.
\newblock In preparation.

\bibitem{deng_de_jeu:2021}
Y.~Deng and M.~de~Jeu.
\newblock Convergence structures and locally solid topologies on vector
  lattices of operators.
\newblock {\em Banach J. Math. Anal.}, 15(3):Paper No. 57, 2021.

\bibitem{deng_de_jeu:2022}
Y.~Deng and M.~de~Jeu.
\newblock Vector lattices with a {H}ausdorff uo-{L}ebesgue topology.
\newblock {\em J. Math. Anal. Appl.}, 505(1):125455, 2022.

\bibitem{gao_leung:2018}
N.~Gao and D.H. Leung.
\newblock Smallest order closed sublattices and option spanning.
\newblock {\em Proc. Amer. Math. Soc.}, 146(2):705--716, 2018.

\bibitem{gao_troitsky_xanthos:2017}
N.~Gao, V.G. Troitsky, and F.~Xanthos.
\newblock Uo-convergence and its applications to {C}es\`aro means in {B}anach
  lattices.
\newblock {\em Israel J. Math.}, 220(2):649--689, 2017.

\bibitem{gao_xanthos:2014}
N.~Gao and F.~Xanthos.
\newblock Unbounded order convergence and application to martingales without
  probability.
\newblock {\em J. Math. Anal. Appl.}, 415(2):931--947, 2014.

\bibitem{kandic_taylor:2018}
M.~Kandi\'{c} and M.A. Taylor.
\newblock Metrizability of minimal and unbounded topologies.
\newblock {\em J. Math. Anal. Appl.}, 466(1):144--159, 2018.

\bibitem{luxemburg_de_pagter:2002}
W.A.J. Luxemburg and B.~de~Pagter.
\newblock Maharam extensions of positive operators and {$f$}-modules.
\newblock {\em Positivity}, 6(2):147--190, 2002.

\bibitem{luxemburg_de_pagter:2005a}
W.A.J. Luxemburg and B.~de~Pagter.
\newblock Representations of positive projections. {I}.
\newblock {\em Positivity}, 9(3):293--325, 2005.

\bibitem{luxemburg_zaanen_RIESZ_SPACES_VOLUME_I:1971}
W.A.J. Luxemburg and A.C. Zaanen.
\newblock {\em Riesz spaces. {V}ol. {I}}.
\newblock North-Holland Publishing Co., Amsterdam-London, 1971.

\bibitem{o'brien_troitsky_van_der_walt_UNPUBLISHED:2021}
M.~O'Brien, V.G. Troitsky, and J.H. van~der Walt.
\newblock Net convergence structures with applications to vector lattices.
\newblock Preprint, 2021. Available at
  \url{https://arxiv.org/pdf/2103.01339.pdf}. To appear in Quaest.\ Math.

\bibitem{synnatzschke:1980}
J.~Synnatzschke.
\newblock \"{U}ber einige verbandstheoretische {E}igenschaften der
  {M}ultiplikation von {O}peratoren in {V}ektorverb\"{a}nden.
\newblock {\em Math. Nachr.}, 95:273--292, 1980.

\bibitem{taylor:2019}
M.A. Taylor.
\newblock Unbounded topologies and {$uo$}-convergence in locally solid vector
  lattices.
\newblock {\em J. Math. Anal. Appl.}, 472(1):981--1000, 2019.

\bibitem{van_imhoff_THESIS:2012}
H.~van Imhoff.
\newblock Order-convergence in partially ordered vector spaces.
\newblock Bachelor's thesis, Leiden University, Leiden, 2012.
\newblock Available at
  \url{https://www.universiteitleiden.nl/binaries/content/assets/science/mi/scripties/vanimhoffbach.pdf}.

\bibitem{wickstead:2017c}
A.W. Wickstead.
\newblock Ordered {B}anach algebras and multi-norms: some open problems.
\newblock {\em Positivity}, 21:817--823, 2017.

\bibitem{zaanen_INTRODUCTION_TO_OPERATOR_THEORY_IN_RIESZ_SPACES:1997}
A.C. Zaanen.
\newblock {\em Introduction to operator theory in {R}iesz spaces}.
\newblock Springer-Verlag, Berlin, 1997.

\end{thebibliography}

\end{document}